\documentclass[12pt]{article}

\usepackage{amsmath,amssymb,amsthm,epsfig,mathrsfs,amsfonts}
\usepackage{hyperref}
\usepackage{authblk}
\usepackage{fullpage,tikz-cd}

\newcommand{\Aut}{\mathop{\rm Aut}\nolimits}
\newcommand{\IMG}{\mathop{\rm IMG}\nolimits}
\newcommand{\Isom}{\mathop{\rm Isom}\nolimits}

\newcommand{\St}{\mathop{\rm St}\nolimits}
\newcommand{\Stg}{\mathop{\rm St}\nolimits_G}
\newcommand{\G}{\mathcal G}
\newcommand{\GE}{\mathcal G_{(01)^\infty}}
\newcommand{\fpG}{\widetilde{\mathcal G}}
\newcommand{\fpB}{\widetilde{\mathcal B}}

\newcommand{\A}{\mathcal A}
\newcommand{\B}{\mathcal B}

\newcommand{\D}{\mathcal D}

\newcommand{\Z}{\mathbb Z}

\newcommand{\Q}{\mathbb Q}
\newcommand{\R}{\mathbb R}
\newcommand{\Sym}{\mathop{\rm Sym}\nolimits}

\newcommand{\Zp}{\Z/(p^n\Z)}

\newcommand{\be}{\mathrm e}

\newcommand{\rootv}{\Lambda}

\newtheorem{theorem}{Theorem}[section]

\newtheorem{corollary}[theorem]{Corollary}
\newtheorem{proposition}[theorem]{Proposition}
\newtheorem{lemma}[theorem]{Lemma}

\newtheorem{question}[theorem]{Question}

\newtheorem{convention}[theorem]{Convention}
\newtheorem{theoremx}{Theorem}

\newtheorem*{theoremhnn}{Theorem~\ref{thmx:hnn}}
\newtheorem*{theoremisoms}{Theorem~\ref{thmx:isoms}}

\theoremstyle{definition}
\newtheorem{definition}[theorem]{Definition}
\newtheorem{example}[theorem]{Example}
\newtheorem{remark}[theorem]{Remark}

\numberwithin{equation}{section}

\title{Liftable self-similar groups and scale groups}

\author[1]{Rostislav Grigorchuk}
\affil[1]{Department of Mathematics\\
        Texas A\&M University\\
        College Station, TX 77843-3368\\
        \href{mailto:grigorch@math.tamu.edu}{grigorch@math.tamu.edu}}

\author[2]{Dmytro Savchuk}
\affil[2]{Department of Mathematics and Statistics\\
        University of South Florida\\
        4202 E Fowler Ave\\
        Tampa, FL 33620-5700\\
        \href{mailto:savchuk@usf.edu}{savchuk@usf.edu}}

\begin{document}

\maketitle

\begin{abstract}
We canonically identify the groups of isometries and dilations of local fields and their rings of integers with subgroups of the automorphism group of the $(d+1)$-regular tree $\widetilde T_{d+1}$, where $d$ is the residual degree. Then we introduce the class of liftable self-similar groups acting on a $d$-regular rooted tree whose ascending HNN extensions act faithfully and vertex transitively on $\widetilde T_{d+1}$ fixing one of the ends. The closures of these extensions in $\Aut(\widetilde T_{d+1})$ are totally disconnected locally compact group that belong to the class of scale groups as defined in~\cite{willis:scale_groups}. We give numerous examples of liftable groups coming from self-similar groups acting essentially freely or groups admitting finite $L$-presentations. In particular, we show that the finitely presented group constructed by the first author in~\cite{grigorch:example} and the finitely presented HNN extension of the Basilica group constructed in~\cite{grigorch_z:basilica_sp} embed into the group $\mathcal D(\Q_2)$ of dilations of the field $\Q_2$ of $2$-adic numbers. These actions, translated to $\widetilde T_3$, are 2-transitive on the punctured boundary of $\widetilde T_3$. Also we explore scale-invariant groups studied in~\cite{nekrashevych_p:scale_invariant} with the purpose of getting new examples of scale groups.
\end{abstract}

\tableofcontents

\section{Introduction}
The main goal of this article is to shed more light on the question about what kind of groups can act by automorphisms and cocompactly on a locally finite infinite tree. As a byproduct of our considerations we gain some information about subgroups of groups of isometries (and more generally of dilations) of local fields, and about the class of scale groups, an important subclass of totally disconnected locally compact (TDLC) groups. Our focus is on subgroups of automorphisms of a $(d+1)$-regular tree ($d\geq 2$) which we will denote by $\widetilde{T}_{d+1}$ (keeping the notation $T_d$ for a $d$-regular rooted tree that also will play an important role in the paper). Such trees are the most symmetric objects among all trees and deserve special attention. The group $\Aut(\widetilde T_{d+1})_{\omega}$ of automorphisms of $\widetilde T_{d+1}$ preserving an end $\omega$, is known for a while, in particular in relation with amenability (of topological groups), random walks, and harmonic analysis~\cite{cartwright_kw:random_walks_on_the_affine_group94,bartholdi_nw:horocyclic_product_of_trees08,woess:rw}.

The study of groups acting on trees (and more generally metric trees, including $\R$-trees) was initiated by J.-P.~Serre and J.~Tits in 1960's and the basics of the theory are known as the \emph{Bass-Serre theory}~\cite{serre:arbres77,bass:covering_theory93,dicks_d:groups_acting_on_graphs89}. One of the key questions is which (discrete or more generally topoplogical) groups act on a tree faithfully and vertex transitively. There are many articles on this topic (see, for example,~\cite{burger_m:groups_acting_on_trees00,bass_l:rigidity94,button:groups_acting_faithfully_on_trees}) but the question still is far from getting the satisfactory answer. A special case considered by Culler and Morgan in~\cite[Theorem 2.7]{culler_m:group_actions_on_R_trees87} (see also a more detailed exposition of this result in~\cite{pays_v:sous-group_libres91}, similar result in connection to closed amenable groups acting on trees by Nebbia~\cite{nebbia:amenability_and_kunze-stein_property88}, and a short exposition with a simple proof in an unpublished note by Wilton~\cite{wilton:group_actions_on_trees}) is when the group $G$ does not contain a free group $F_2$ on two generators as a subgroup. In this case any action of $G$ on a tree must preserve a vertex, an edge, an end or a pair of ends (as a set). Discrete amenable groups do not contain $F_2$ as a subgroup, hence when acting on a tree, they must fix something. The case when there is a global fixed vertex is orthogonal to co-compactness of the action and is a matter of study of groups acting on rooted trees~\cite{gns00:automata,nekrash:self-similar}. This case is quite understood and the groups in this class are residually finite~\cite{gns00:automata,button:groups_acting_faithfully_on_trees}. The case when the edge is fixed is not far from the previous one, and the case when a group fixes a set of two ends is also well understood. For instance, when it fixes both ends then the involved group is a semidirect product of $\Z$ and a group acting level transitively on a rooted tree.

Thus we are left with the case when $G$ fixes an end of the tree (i.e., a point of the boundary of the tree) and we focus our attention to this case.

If a field is supplied with a natural metric then one could be interested in the group of its isometries (or more generally dilations). A natural class of fields to be considered is the class of non-Archimedian local fields which, by the famous characterization (see, for example, Theorem~22 in~\cite{pontryagin:topological_groups_3rd_edition_translated86} or Remark~7.49 in~\cite{milne:ANT}, or an earlier work~\cite{otobe:locally_compact_fields45}), are either finite extensions of the field $\Q_p$ of $p$-adic numbers or are isomorphic to the field $\mathbb F_q((x))$ of Laurent powers series over a finite field $\mathbb F_q$.

Our result (in fact observation) given by Theorem~\ref{thmx:isoms} and numerous examples coming from theory of groups acting on rooted trees show that the groups of isometries and dilations of local fields contain a plethora of exotic subgroups satisfying various conditions of transitivity, including those that do not satisfy the famous Tits alternative that states that a finitely generated linear group (i.e., a group generated by matrices over a field) either contains a free subgroup $F_2$ or is virtually solvable, and hence amenable. This includes groups of intermediate growth, finitely presented amenable, but not elementary and not subexponentially amenable groups, non-elementary just-infinite groups, groups of branch type, and finitely constrained groups.
Thus, the situation here is very different from the situations with subgroups of linear groups.

While the theory of connected locally compact topological groups is well developed, totally disconnected locally compact groups are very far from being well understood. A mysterious connection of them to self-similar groups was discovered recently in~\cite{caprace_w:tdlc_invitation,willis:scale_groups} via the notion of \emph{scale} group, that is a closed vertex transitive subgroup of $\Aut(\widetilde{T}_{d+1})$ fixing an end. By the result of Nebbia~\cite[Theorem~2]{nebbia:amenability_and_kunze-stein_property88} this is exactly the class of closed amenable (as topological groups) vertex transitive subgroups of $\Aut(\widetilde{T}_{d+1})$.

Our main result, Theorem~\ref{thmx:hnn}, gives a way to construct new examples of scale groups realizing them as closures of the ascending HNN extensions of the so-called liftable groups introduced in the article (see Definition~\ref{def:liftable}).

The alternative approach of getting scale groups is via not finitely co-hopfian groups, i.e., groups containing proper subgroups of finite index isomorphic to them. Getting interesting examples of finitely generated groups of this sort is a challenging problem which is closely related to the problem of constructing \emph{scale-invariant} groups (see Definition~\ref{def:scale_invariant}) initiated in~\cite{nekrashevych_p:scale_invariant}. The level of transitivity of the scale group is a matter of a special interest and is discussed for instance in~\cite{willis:scale_groups}. We give some insight on it as well in Theorem~\ref{thmx:grig}.

Scale groups are amenable groups in the category of topological groups and in most of our examples they are closures of (non-elementary) amenable abstract groups. Moreover, in some cases these amenable groups are finitely presented and non-elementary amenable. 

We give also examples of scale groups that are closures of non-amenable groups (even not satisfying the Haagerup property). Observe that closed subgroups of $\Aut(\widetilde{T}_{d+1})$ that act transitively on $\partial\widetilde{T}_{d+1}$ and hence are not scale groups belong to the class of Kunze-Stein groups~\cite{nebbia:amenability_and_kunze-stein_property88}.

Now let us pass more closely to the context of the article. We start from establishing a connection between groups of isometries and dilations of local fields and groups of automorphisms of regular trees. The following theorem could be viewed as a folklore type statement, but it seems that it was never stated in this form before. By dilation of a metric space we mean a homeomorphism that uniformly expands distances by some constant $\lambda>0$ (see Definition~\ref{def:dilation}).


\begin{theoremx}
\label{thmx:isoms} Let $F$ be a (non-Archimedean) local field with the ring of integers $\mathcal O$ and the residue field of size $q$, endowed with the metric induced by its valuation. Then\\[-5mm]
\begin{itemize}
\item[(i)] the group $\Isom(\mathcal O)$ of isometries of $\mathcal O$ is isomorphic to the group $\Aut(T_q)$ of automorphisms of $T_q$.
\item[(ii)] the group $\Isom(F)$ of isometries of $F$ is isomorphic to the group $\Aut_0(\widetilde T_{q+1})_{\omega}$ consisting of automorphisms of $\widetilde T_{q+1}$ that fix a distinguished end $\omega$ pointwise (eventually fixing the vertices in the ray corresponding to $\omega$).
\item[(iii)] the group $\D(F)$ of dilations of $F$ is isomorphic to the group $\Aut(\widetilde T_{q+1})_{\omega}$ consisting of automorphisms of $\widetilde T_{q+1}$ that fix an end $\omega$ as a point of the boundary of $\widetilde T_{q+1}$ (but possibly moving the vertices in rays corresponding to $\omega$ along those rays while preserving $\omega$ as a point of the boundary).
\end{itemize}
\end{theoremx}


The automorphism groups involved in the theorem were discussed before (see, for example,~\cite{cartwright_kw:random_walks_on_the_affine_group94,bartholdi_nw:horocyclic_product_of_trees08}) and we present some background about them in Subsection~\ref{ssec:history_trees}.

As was already mentioned one of the goal of the paper is to construct examples of amenable totally disconnected locally compact (TDLC) groups coming from the actions of self-similar groups on $d$-regular rooted tree $T_d$, $d\geq 2$. Such groups are built as the closures of embeddings of the HNN extensions of \emph{liftable self-similar groups} defined in the article (see Definition~\ref{def:liftable}) and scale-invariant groups (see Definition~\ref{def:scale_invariant}), in the group of automorphisms of $\widetilde T_{d+1}$. Obtained groups belong to the class of scale groups. Scale groups play an important role in theory of TDLC groups as shown in works of P.-E.~Caprace and G.~Willis~\cite{caprace_w:tdlc_invitation,willis:scale_groups}. We comment on the connection of our results to the work of Willis~\cite{willis:scale_groups} in Subsection~\ref{ssec:connection_to_willis}.

The idea of lifting (the notion that will appear in the next Theorem~\ref{thmx:hnn}) has its roots in the ``trick'' used in~\cite{grigorch:example} to embed the group $\G$ of intermediate growth constructed by the first author~\cite{grigorch:burnside} into a finitely presented group. The presentation of $\G$
\begin{equation}
\label{eqn:presentationLys}
\G=\langle a,b,c,d\mid a^2, b^2, c^2, d^2, bcd, \sigma^i((ad)^4),\sigma^i((adacac)^4), i\geq0\rangle,
\end{equation}
where the substitution $\sigma$ is defined as
\begin{equation}
\label{eqn:sigma_grigorchuk}
\sigma\colon\left\{
\begin{array}{l}
a\mapsto aca\\
b\mapsto d\\
c\mapsto b\\
d\mapsto c
\end{array}\right.
\end{equation}
discovered by I.~Lysionok in~\cite{lysionok:presentation}, leads to the embedding of $\G$ into the ascending HNN extension $\fpG$
\begin{equation}
\label{eqn:presentation}
\fpG=\langle a,b,c,d,t\mid a^2, b^2, c^2, d^2, bcd, (ad)^4,(adacac)^4,tst^{-1}\sigma(s)^{-1}, s\in\{a,b,c,d\}\rangle.
\end{equation}
The latter is a finitely presented amenable, but not elementary amenable group containing a subgroup of intermediate growth. The presentation~\eqref{eqn:presentation} was later simplified by L.~Bartholdi (given in~\cite{harpe_cg:paradoxical}) to the one with only 2 generators and 4 relators of the combined length of 109:
\begin{equation}
\label{eqn:presentation_short}
\fpG\cong\left\langle a,t \mid a^2=TaTatataTatataTataT=(Tata)^8=(T^2ataTat^2aTata)^4=1\right\rangle,
\end{equation}
where $T$ denotes $t^{-1}$.

The further observation made by the first author a long time ago and privately reported to A.~Valette is that $\fpG$ embeds into the group of automorphisms of $\widetilde T_3$ preserving one end of this tree. Hence we have a good example illustrating the theorem of Nebbia~\cite{nebbia:amenability_and_kunze-stein_property88} mentioned earlier. Moreover, as recently shown by C.~Bodart this group does not have rational cross-sections~\cite{bodart:rat_cross-sections}, which is the first example of such a group in the class of finitely presented groups with solvable word problem.

The presentation found by Lysionok attracted a lot of attention and led to the definition (in different forms) of (finite) $L$-\emph{presentation} (also sometimes called \emph{endomorphic presentation})~\cite{bar_gs:branch,bartholdi:epimorphic03,grigorch_ss:img}. It was shown in~\cite{gr99:schur} that the Schur multiplier $H_2(\G,\Z)$ of $\G$ is isomorphic to $(\Z/2\Z)^\infty$ and that the relators in presentation~\eqref{eqn:presentationLys} are independent. This also gave an answer to G.~Bumslag's question raised in~\cite{baumslag:topics_book93}.

The idea of self-similarity that is in the root of the construction of $\G$ got remarkable developments in various directions, including theory of automata groups~\cite{gns00:automata,bondarenko_gkmnss:full_clas32_short}, iterated monodromy groups~\cite{nekrash:self-similar}, combinatorics~\cite{grigorchuk-s:hanoi-cr}, analysis on graphs~\cite{bartholdi_g:spectrum,grigorch_s:hanoi_spectrum,grigorch_z:lamplighter,grigorch_ss:img}, computer science~\cite{gns00:automata,miasnikov_s:cayley_automatic11,miasnikov_s:automatic_graph}, cryptography~\cite{grigorch_g:key_agreement19,myasnikov_u:random_subgroups08,myasnikov_su:non_commutative_crypto_book11,garzon_z:crypto,petrides:cryptoanalysis_grigorchuk,kahrobaei_ms:lba} and coding theory~\cite{cull_n:hanoi_codes99,grigorchuk-s:hanoi-cr}, random walks~\cite{kaimanovich:self-similarity_and_random_walks09,bkn:amenab,amir_av:amenab_linear_autom}. It came also to the theory of TDLC groups as follows from the articles~\cite{caprace_w:tdlc_invitation,willis:scale_groups}. Self-similar groups related to the scale groups have an additional property of being \emph{self-replicating} (see Definition~\ref{def:self-replicating}). Such groups appear quite often among self-similar groups as one can conclude from classification result on groups generated by 3-state invertible automata acting over a 2-letter alphabet~\cite{bondarenko_gkmnss:full_clas32_short}.

The question of interest is which of self-replicating groups lead to examples of scale groups. This question is studied from one angle in~\cite{willis:scale_groups} and represents one of the interests of this work. In Section~\ref{ssec:connection_to_willis} we make a comment on the relation between results of this paper and of~\cite{willis:scale_groups}. Our next result is based on the notion of a liftable self-similar group and a homomorphism of lifting $\sigma\colon G\to\Stg(i)$, where $i$ is one of the vertices of the first level of the tree on which the group $G$ acts, and $\Stg(i)$ is the stabilizer of this vertex in $G$. An example of such lifting is given by substitution~\eqref{eqn:sigma_grigorchuk}, which extends to the injective endomorphism of $\G$. To unify the notation, we will use the same letter $\sigma$ for the liftings in most of the groups that we deal with, unless it is important to emphasize the group that we are dealing with. In the result below $\rootv$ denotes the ``leftmost'' vertex of level 0 in $\widetilde T_{d+1}$ (as shown in Figure~\ref{fig:biinfinite_tree}) and $\pi_{\rootv}\colon \St_{\Aut(\widetilde T_{d+1})}(\rootv)\to\Aut(T_d)$ is a homomorphism that sends each element to its section (or projection) at the vertex $\rootv$, that is an automorphism of a rooted tree $T_d$ with the root $\Lambda$.

\begin{theoremx}
\label{thmx:hnn}
Let $G$ be a liftable group acting on $T_d$ and $\sigma\colon G\to\Stg(i)$ be the corresponding lifting. Then
\begin{itemize}
\item[(i)] There is an embedding $\theta$ of the ascending HNN extension $\widetilde G$ of $G$ by $\sigma$
\begin{equation}
\label{eqn:hnn_pres}
\widetilde G=G*_\sigma=\langle G,t\mid \text{relations in}\ G, tgt^{-1}=\sigma(g)\ \text{for all}\ g\in G\rangle
\end{equation}
into $\Aut(\widetilde T_{d+1})_{\omega}$ for the end $\omega=0^{-\infty}$ of $\widetilde T_{d+1}$.
\item[(ii)] If $G$ acts transitively on the first level of $T_d$, then $G$ is self-replicating, $\theta(\widetilde G)$ acts transitively on the set of vertices of $\widetilde T_{d+1}$, and the closure $W$ of $\theta(\widetilde G)$ in $\Aut(\widetilde T_{d+1})_{\omega}$ is a scale group.
\item[(iii)] The projection $\pi_{\rootv}$ maps $\St_{\theta(\widetilde G)}(\rootv)$ onto $G$ and $\St_{\theta(\widetilde G)}(\rootv)=\bigl(\ker\pi_{\rootv}\cap\theta(\widetilde G)\bigr) \rtimes \theta(G)$.
\item[(iv)] The projection $\pi_{\rootv}$ maps $\St_{W}(\rootv)$ onto $\overline{G}$ and $\St_{W}(\rootv)=\bigl(\ker\pi_{\rootv}\cap W\bigr)\rtimes \theta(\overline{G})$.
\end{itemize}
\end{theoremx}

There are several sources of examples when this theorem is applicable. The first one is the class of self-similar groups acting essentially freely on the boundary of rooted tree. It includes such groups as the lamplighter group $\Z_2\wr\Z$, some of Baumslag-Solitar groups, realizations of free groups and the free products of cyclic groups as self-similar groups~\cite{vorobets:aleshin,vorobets:series_free,savchuk_v:free_prods}. All such groups generated by 3-state automata over 2-letter alphabet are completely classified in~\cite{grigorch_s:essfree}.

Another source of examples consists of groups of branch type admitting finite $L$-presentations satisfying an additional condition that the substitution $\sigma$ involved in the presentation induces a lifting. This class includes the group $\G$, Gupta-Sidki 3-group~\cite{gupta_s:burnside}, Basilica group $\B$~\cite{grigorch_z:basilica}, etc. In some cases, like GGS-groups~\cite{bar_gs:branch} and some of $\G_{\omega}$ groups~\cite{grigorch:degrees}, liftability can be verified directly without the knowledge of a group presentation. It follows from~\cite{grigorch_z:basilica,bartholdi_v:amenab} that the Basilica group is an example of an amenable but not subexponentially amenable group (i.e., whose amenability does not come in any way from groups of subexponential growth). That was the first example of such a group, whose existence was questioned in~\cite{grigorch:hilbert}. As pointed out in~\cite{grigorch_z:basilica_sp}, using the simplified $L$-presentation of $\B$ constructed by Bartholdi, the ascending HNN extension $\fpB$ of Basilica group $\B$ defined in Example~\ref{ex:basilica}, similarly to $\fpG$, is a finitely presented group. Also in~\cite{grigorch_z:basilica_sp} it was shown that $\B$ (and hence $\fpB$) does not belong to the class $SG$ of subexponentially amenable groups introduced in~\cite{grigorch:example} and the question about amenability of these groups was raised.  It got a positive answer by Barthlodi and Virag in~\cite{bartholdi_v:amenab}. More than a half of our examples belong to the class of branch groups introduced by the first author in~\cite{grigorch:jibranch,grigorch:branch}.

As one of the main applications of Theorem~\ref{thmx:hnn} we establish the following embedding result.

\begin{theoremx}
\label{thmx:grig}
The groups $\fpG$ and $\fpB$ embed into $\D(\Q_2)$. Moreover, when $\D(\Q_2)$ is identified with $\Aut(\widetilde T_{3})_{\omega}$, the images of these embeddings act transitively on $\widetilde T_3$ and their closures are scale groups acting 2-transitively on the punctured boundary $\partial_{\omega}\widetilde T_{3}$ of $\widetilde T_{3}$.
\end{theoremx}

It is shown in~\cite{willis:scale_groups} that every scale group acts transitively on the corresponding punctured boundary $\partial_{\omega}\widetilde T_{d+1}$. The 2-transitivity condition in the above theorem is a stronger statement.

Finally, at the end of the article we discuss one more method of constructing scale groups. It is based on the use of some facts from the Bass-Serre theory of groups acting on trees and not finitely co-Hopfian groups, i.e., groups that contain a subgroup of finite index isomorphic to the whole group. The class of scale-invariant groups studied in~\cite{nekrashevych_p:scale_invariant} is suitable for our purposes and allows us to construct more examples of scale groups. One such example is obtained as a closure of a group without Haagerup property. Observe that the closure of a group with Kazhdan's property (T) cannot be a scale group since this property implies Serre's property (FA)~\cite{watatani:kazhdan_implies_fa82}, which means that every action of a group on a tree has a global fixed point. Nevertheless we use property (T) in Section~\ref{sec:haagerup} to construct groups without Haagerup property whose closures are scale groups. Thus (and we emphasize it), the examples of scale groups given in this article are of two kinds: obtained as the closures of (elementary and non-elementary) amenable groups and as the closures of non-amenable groups.

The paper is organized as follows. We begin with preliminaries in Section~\ref{sec:pre}. In Section~\ref{sec:isometries} we describe the identification of the groups of isometries and dilations of local fields and their rings of integers with groups of automorphisms of regular rooted and unrooted trees. Section~\ref{sec:embedding} contains the definition of liftable groups and the main embedding result. In Section~\ref{sec:examples} we provide many examples of liftable groups and analyze the actions of their HNN extensions on corresponding regular trees. Section~\ref{sec:bass} offers an alternative method of constructing of actions of ascending HNN extentions on regular trees via the Bass-Serre theory (see Theorem~\ref{thm:bass_serre}). Finally, in Section~\ref{sec:haagerup} we use some scale-invariant groups to provide new examples of a scale groups. Note that the paper contains a rather extended bibliography as we tried to unite various notation and approaches used in the literature.

\noindent{\textbf{Acknowledgements.} The authors would like to thank A.~Garrido, A.~Minasyan, V.~Nekrashevych, A.~Valette, and G.~Willis for fruitful discussions that enhanced the paper. The work on this project was partially conducted during the authors' visits to American Institute of Mathematics via the SQuaRE program; authors thank the Institute for hospitality and support. The first author was supported by the Deutsche Forschungsgemeinschaft (DFG, German Research Foundation) -- SFB-TRR 358/1 2023-491392403,  Humboldt Foundation and expresses his acknowledgement to the University of Bielefeld.   Also  he   is  supported  by  Travel Support for Mathematicians  grant  MP-TSM-00002045  from Simons Foundation. Finally, the authors appreciate the support from University of Geneva (Swiss NSF grant
200020-200400), Texas A\&M University and University of South Florida during working on the paper.}

\section{Preliminaries}

\label{sec:pre}
We start this section by introducing the notions and terminology of automorphisms of regular rooted trees, self-similar groups, and transformations generated by Mealy automata. For more details, the reader is referred to~\cite{gns00:automata,nekrash:self-similar}.

\subsection{Rooted trees and self-similar groups}

Let $X=\{0,1,\ldots ,d-1\}$ be a finite alphabet with $d\geq 2$ elements (called letters) and let $X^*$ denote the set of all finite words over $X$. The set $X^*$ can be equipped with the structure of a rooted $d$-ary tree $T(X)$ by declaring that $v$ is adjacent to $vx$ for every $v\in X^*$ and $x\in X$. As a graph the tree $T(X)$ is isomorphic to $T_d$, but the vertices of $T(X)$ are labelled by finite words over $X$. We will denote the length of a word $w\in X^*$ by $|w|$. The empty word corresponds to the root of the tree and for each positive integer $n$ the set $X^n$ corresponds to the $n\,$th level of $T(X)$. Also the set $X^\infty$ of infinite (to the right) words over $X$ can be identified with the \emph{boundary} $\partial T(X)$ of the tree $T(X)$, which consists of all infinite paths in the tree, without backtracking, initiating at the root. We consider automorphisms of the tree $T(X)$ (that is, the bijections of $T(X)$ that preserve the root and the adjacency of vertices). We sometimes denote the tree $T(X)$ as $T_{|X|}$ when only the cardinality of $X$ is important. The group of all automorphisms of $T(X)$ is denoted by $\Aut(T(X))$. To operate with such objects, we use alternatively the languages of sections and wreath recursions, and of Mealy automata.

Let $g\in\Aut(T(X))$ and $x\in X$. For any $v\in X^*$ we can write \[g(xv)=g(x)v'\] for some $v'\in X^*$. Then the map $g|_x\colon X^*\to X^*$ given by \[g|_x(v)=v'\] defines an automorphism of $T(X)$ which we call the \emph{section} (or \emph{state}) of $g$ at vertex $x$. We can inductively extend the definition of a section at a letter $x\in X$ to a section at any vertex $x_1x_2\ldots x_n\in X^*$ as follows:
\[g|_{x_1x_2\ldots x_n}=g|_{x_1}|_{x_2}\cdots|_{x_n}.\]
With this notation for any $g\in\Aut(T(X))$, and any words $v,w\in X^*$ we have
\begin{equation}
\label{eqn:section_words}
g(vw)=g(v)g|_v(w).
\end{equation}

\begin{definition}
A subgroup of $\Aut(T(X))$ is called \emph{self-similar} if it is closed under taking sections at the vertices of $T(X)$.
\end{definition}


We adopt the following convention throughout the paper.
\begin{convention}
\label{conv:product}
If $g$ and $h$ are elements of some group acting on a set $Y$ and $y\in Y$, then
$$gh(y)=h(g(y)).$$
\end{convention}
We work with the right actions partially because it corresponds to the standard order of multiplication of permutations in \verb"GAP", the computer algebra platform containing the package \verb"AutomGrp"~\cite{muntyan_s:automgrp} dealing with self-similar groups and semigroups and used in this paper for many routine calculations. However, we due convenience in notation, we still write the arguments of functions on the right. Note that knowledge of GAP is not required for reading this paper or for verification of its results.

With the above convention, the state $g|_v$ at $v\in X^*$ of any product $g=g_1g_2\cdots g_n$, where $g_i\in\Aut(T(X))$ for $1\leq i\leq n$, can be computed as follows:
$$g|_v=g_1|_v g_2|_{g_1(v)}\cdots
g_n|_{g_1g_2\cdots g_{n-1}(v)}.$$

Also we use the language of wreath recursions. For each self-similar group $G$ there is a natural embedding
\[G\hookrightarrow G \wr \Sym(X),\]
where $\wr$ is a symbol for the permutational wreath product, and $\Sym(X)$ denotes the symmetric group on $X$. This embedding is given by
\begin{equation}
\label{eq:wreath}
G\ni g\mapsto (g_0,g_1,\ldots,g_{d-1})\sigma_g\in G\wr \Sym(X),
\end{equation}
where $g_0,g_1,\ldots,g_{d-1}$ are the states of $g$ at the vertices of the first level, and $\sigma_g$ is the transformation of $X$ induced by the action of $g$ on the first level of the tree. If $\sigma_g$ is the trivial permutation, it is customary to omit it in~\eqref{eq:wreath}. We call $(g_0,g_1,\ldots,g_{d-1})\sigma_g$ the \emph{decomposition of $g$ at the first level} (or the \emph{wreath recursion of $g$}). When this does not cause confusion, we identify $g$ with its wreath recursion and write simply
\begin{equation}
\label{eqn:decomposition}
g=(g_0,g_1,\ldots,g_{d-1})\sigma_g.
\end{equation}

The decomposition at the first level of all generators $\A_q$ of an automaton group $\mathbb{G}(\A)$ under the embedding~\eqref{eq:wreath} is called the \emph{wreath recursion defining the group}. Such a decomposition is especially convenient for computing the states of group elements. Indeed, the products and inverses of automorphisms can be found as follows. If $g=(g_0,g_1,\ldots,g_{d-1})\sigma_g$ and $h=(h_0,h_1,\ldots,h_{d-1})\sigma_h$ are two elements of $\Aut(T(X))$, then $$gh=\left(g_0h_{\sigma_g(0)},g_1h_{\sigma_g(1)},\ldots,g_{d-1}h_{\sigma_g(d-1)}\right)\sigma_g\sigma_h$$ and the wreath recursion of $g^{-1}$ is
$$g^{-1}=\left(g^{-1}_{\sigma_g^{-1}(0)},g^{-1}_{\sigma_g^{-1}(1)},\ldots,g^{-1}_{\sigma_g^{-1}(d-1)}\right)\sigma_g^{-1}.$$
Note that the order of multiplication here is compatible with Convention~\ref{conv:product}. Some authors choose to use the left actions and they write permutations on the left side in wreath recursions. These approaches are equivalent and there is an easy translation between them: every product has to be reversed everywhere.

Equation~\eqref{eqn:section_words} for an automorphism $g$ given by decomposition~\eqref{eqn:decomposition}, and for $v=x\in X$ then can be written as:
\[g(xw)=\sigma_g(x)g_x(w).\]

For a group $G$ acting on $T(X)$ and $v\in X^*$ the stabilizer $\Stg(v)$ of vertex $v$ is the subgroup of $G$ consisting of all elements that fix $v$. For each $m\geq 1$ the stabilizer $\Stg(X^m)$ of level $m$ in $G$ is the subgroup of $G$ consisting of all elements that fix all vertices of level $m$. Stabilizers of levels are normal finite index subgroups of $G$ such that
\[\cap_{m\geq1}\Stg(X^m)=\{1\}.\]
For each $v\in X^*$ the map
\begin{equation}
\label{eqn:proj}
\begin{array}{llll}\pi_v\colon &\St_{\Aut(T(X))}(v)&\to&\Aut(T(X))\\
&g&\mapsto&g|_v
\end{array}
\end{equation}
defines a homomorphism that we will call projection. For each subgroup $G$ of $\Aut(T(X))$ homomorphism $\pi_v$ restricts to a homomorphism $\Stg(v)\to\Aut(T(X))$. Moreover, if $G$ is a self-similar group, then since $G$ is closed under taking the sections, $\pi_v$ is a homomorphism from $\Stg(v)$ to $G$. In the case $X=\{0,1,\ldots,d-1\}$, $\pi_i$, $i=0,1,\ldots,d-1$, denotes the corresponding projection.

An important subclass of self-similar groups is the class of self-replicating groups.

\begin{definition}
\label{def:self-replicating}
A self-similar group is called \emph{self-replicating} if it acts transitively on $X$ and for each $v\in X^*$ the map $\pi_v\colon\Stg(v)\to G$ is onto.
\end{definition}

It is easy to see that in the above definition we can consider only projections $\pi_v$ for $v\in X$.

It follows from the definition by induction on the level that self-replicating groups act transitively on all levels of $T(X)$.

We can visualize the action of an automorphism $g\in\Aut(T(X))$ using the notion of the \emph{portrait} of $g$, denoted by $\mathcal P(g)$. It is a labeled infinite rooted $|X|$-ary tree with the root labeled by $g$ and each vertex $v$ labeled by $g|_v$. Under each vertex $v$, we write the name of the mapping that $g|_v$ defines on the first level of the subtree hanging from $v$. In the case of a rooted binary tree, we draw an arc (called switch) connecting the two edges hanging down from $v$ if $h|_v$ acts nontrivially on the first level of the subtree hanging from $v$. If there is no switch, it means the action is trivial. Four examples of portraits corresponding to the generators of the group $\G$ are shown in Figure~\ref{fig:grig_gens}.

\subsection{Mealy Automata}
Another common language to describe self-similar groups is the language of Mealy automata.
\begin{definition}
A \emph{Mealy automaton} (or simply \emph{automaton}) is a 4-tuple \[(Q,X,\delta,\lambda),\]
where
\begin{itemize}
\item $Q$ is a set of states
\item $X$ is a finite alphabet (not necessarily $\{0,1,\ldots,d-1\}$)
\item $\delta\colon Q\times X\to Q$ is the \emph{transition function}
\item $\lambda\colon Q\times X\to X$ is the \emph{output function}.
\end{itemize}
If the set of states $Q$ is finite, the automaton is called \emph{finite}. If for every state $q\in Q$ the output function $\lambda_q(x)=\lambda(q,x)$ induces
a permutation of $X$, the automaton $\A$ is called \emph{invertible}. Selecting a state $q\in Q$ produces an \emph{initial automaton} $\A_q$, which formally is a $5$-tuple $(Q,X,\delta,\lambda,q)$.
\end{definition}

Here we consider automata with the same input and output alphabets.

Automata are often represented by their \emph{Moore diagrams}. The Moore diagram of automaton $\A=(Q,X,\delta,\lambda)$ is a directed graph in which the vertices are in bijection with the states of $Q$ and the edges have the form $q\stackrel{x|\lambda(q,x)}{\longrightarrow}\delta(q,x)$ for $q\in Q$ and $x\in X$. For example, Figure~\ref{fig:grig_aut} shows the Moore diagram of the automaton $\A$ that, as is explained later, generates the group $\G$.

\begin{figure}
\begin{center}
\epsfig{file=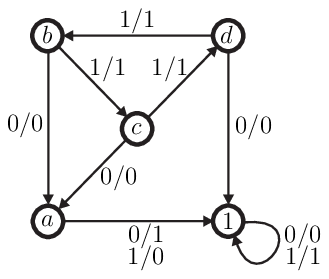}
\end{center}
\caption{The automaton generating the group $\G$\label{fig:grig_aut}}
\end{figure}

Every invertible initial automaton $\A_q$ induces an automorphism of $T(X)$, which is also denoted by $\A_q$, defined as follows. Given a word
$v=x_1x_2x_3\ldots x_n\in X^*$, it scans the first letter $x_1$ and outputs $\lambda(q,x_1)$. The rest of the word is handled similarly by the initial automaton $\A_{\delta(q,x_1)}$. So we can actually extend the functions $\delta$ and $\lambda$ to $\delta\colon Q\times X^*\to Q$ and $\lambda\colon Q\times X^*\to X^*$ via the equations
\[\begin{array}{l}
\delta(q,x_1x_2\ldots x_n)=\delta(\delta(q,x_1),x_2x_3\ldots x_n),\\
\lambda(q,x_1x_2\ldots x_n)=\lambda(q,x_1)\lambda(\delta(q,x_1),x_2x_3\ldots x_n).\\
\end{array}
\]

\begin{definition}
The group generated by all states of an invertible automaton $\A$ viewed as automorphisms of the rooted tree $T(X)$ under the operation of composition is called an \emph{automaton group}.
\end{definition}

In the definition of the automaton, we do not require the set $Q$ of states to be finite. With this convention, the notion of an automaton group is equivalent to the notions of \emph{self-similar group}~\cite{nekrash:self-similar} and \emph{state-closed group}~\cite{nekrash_s:12endomorph}. However, most of the interesting examples of automaton (semi)groups are finitely generated (semi)groups defined by finite automata.

\subsection{The boundary of a rooted tree as a topological and measure space}
The boundary $\partial T(X)$ of the rooted tree is endowed with a natural topology in which two infinite words are close if they have a large common prefix. Each automorphism of $T(X)$ naturally induces a continuous transformation (respectively, homeomorphism) of $\partial T(X)$.

The boundary $\partial T(X)$ can be viewed as an (ultra)metric space in which the distance between two points is equal to $\lambda_k$, where $\{\lambda_n\}$ is any monotonically decreasing sequence converging to 0, and $k$ is the length of the longest common beginning of these two points.
This metric defines a topology on $\partial T(X)$ that coincides with the Tychonoff product topology on $X^{\mathbb N}$ (when using the discrete topologies on each copy of $X$. In this topology $\partial T(X)$ is homeomorphic to the Cantor set. Let $\mu$ be the uniform Bernoulli measure defined by
\[\mu(C_v)=\frac1{|X|^n},\]
where for a vertex $v\in X^n$ the cylindrical set $C_v\subset\partial T(X)$ consists of all infinite paths in $\partial T(X)$ that go through $v$. Such sets generate a $\sigma$-algebra of Borel sets and hence measure $\mu$ is defined by its values on them. This measure is invariant with respect to the action of $\Aut(T(X))$. A action of a countable group $G<\Aut(T(X))$ on $(\partial T(X),\mu)$ by measure preserving transformations is called essentially free if the measure of a set of points in $X$ that have non-trivial stabilizers is equal to 0. The examples of such groups will be given in Section~\ref{sec:examples}.

It is a well known fact (see, for example,~\cite{gns00:automata}) that if a group $G<\Aut(T(X))$ acts transitively on levels of $T(X)$, then its action on $\partial T(X)$ is ergodic and its closure in $\Aut(T(X))$ acts transitively on $\partial T(X)$. More generally, the orbit structure of the action of $G$ on the levels of $T(X)$ gives rise to the decomposition into ergodic components of the action of $G$ on $\partial T(X)$~\cite{grigorch_s:ergodic_decomposition}.

\subsection{Historical remarks}
\label{ssec:historical}
The reader should be aware that during the period 1999--2003, when there was an explosion of the activity on self-similar and branch groups, the terminology associated with the groups acting on rooted trees and, in particular, with self-similar groups was changing. Now it is more or less settled, but still there are nuances. In the articles~\cite{gr99:schur,grigorch:branch,grigorch_hz:profinite_completions00,gns00:automata} the terms ``self-similar'', ``self-replicating'', ``branch'', ``regular branch'', etc. already appear, but sometimes with different meaning. For instance, the term ``self-similar'' action of a group was used in the past to denote what is now called ``self-replicating'' (or ``recurrent'' in~\cite{nekrash:self-similar}).

Currently most of the authors use the definition of a self-similar group presented in~\cite{bartholdi_gn:fractal,nekrash:self-similar} as used in this article, which we will call \emph{standard}.  However, an early version of~\cite{willis:scale_groups} by G.~Willis used a slightly different definition, calling a group self-similar if the projection of the stabilizer of every vertex in the tree on the subtree rooted at that vertex remains in the original group. This condition corresponds to the definition of a \emph{fractal} group in~\cite[Definition~2.1]{bartholdi_g:spectrum}. The following remarks aim to clarify the distinction between these two definitions. Any automorphism $g$ acting on the binary rooted tree with the wreath recursion $g=(h,h^{-1})\varepsilon$ for arbitrary $h\in\Aut(T_2)$ would generate a fractal group isomorphic to $\Z/2\Z$ as the stabilizer of every vertex is trivial. Clearly, this group would not be self-similar (in the standard sense) if $h$ is not a power of $g$.

Another example constructed by A.~Garrido shows that there are self-replicating groups that are fractal, but not self-similar (in the standard sense). Consider the infinite dihedral group $D_\infty$ acting on a binary rooted tree and generated by the following wreath recursion:
\[\begin{array}{rcl}
a&=&(1,1)\varepsilon,\\
b&=&(a,b).
\end{array}
\]
Let $H=\langle ab\rangle\cong\Z$ be the index 2 subgroup in $D_\infty$. Then, since $ab=(b,a)\varepsilon$ we get
\[(ab)^{2^n}=\bigl((ba)^{2^{n-1}},(ab)^{2^{n-1}}\bigr)\]
for $n\geq 1$ and the stabilizer $\St_H(v)$ in $H$ of any vertex $v$ of level $n$ is generated by $(ab)^{2^n}$. Therefore, $\pi_v\colon\St_H(v)\to H$ projects $\St_H(v)$ onto $H$ and $H$ is self-replicating and fractal. However, clearly $H$ is not self-similar in the standard sense because $(ab)|_0=b\notin H$.

On the other hand, the following proposition shows that every fractal group is conjugate to a self-similar group (in the standard sense), so no difference between the who classes can be detected by analyzing the algebraic structure of these groups. Note, that this small difference does not usually play a significant role, but it is important to understand it.

\begin{proposition}[V.~Nekrashevych]
Let $G$ be a group acting on the tree $X^*$ so that for every $g\in G$ stabilizing
a vertex $v$ we have $g|_v\in G$. Then the action of $G$ is conjugate to a self-similar
action.
\end{proposition}

\begin{proof}
Choose a representative in each orbit of the first level, and choose for every
$x\in X$ an element $r_x\in G$ moving the representative $x'$ of the orbit of $x$ to $x$. Denote
$h_x = r_x|_{x'}$. Conjugate then the action of $G$ by the automorphism $a$ of $X^*$
defined by the recursion:
\[a(xv) = x(h_xa)(v).\]

Then $xv = a^{-1}(x(h_xa)(v))$, so, replacing $(h_xa)(v) = w$, we get $a^{-1}(xw) = xa^{-1}h^{-1}_x(w)$ and for every $g\in G$ we have
\[
a^{-1}ga(x) = a^{-1}g(x(h_xa)(v))
= a^{-1}g(x)(g|_xh_xa)(v)
= g(x)(a^{-1}h^{-1}_{g(x)}g|_xh_xa)(v).
\]
The element $h^{-1}_{g(x)}g|_xh_x$ is equal to the section of $r^{-1}_xgr_x$ at the representative $x'$ of the orbit of x, since we have $r_x(x')=x$, $r_x|_{x'}=h_x$, $r_{g(x)}=r_x$. The element  $r^{-1}_xgr_x$ stabilizes $x'$, so its section at $x'$ belongs to $G$. In the conjugated action, the section of $g$ at $x$ is equal to $(r^{-1}_xgr_x)|_{x'}$, hence belongs to $G$.
\end{proof}

Additionally, other terms have been used to denote the class of self-similar groups, like ``automaton groups'' and ``state-closed groups'' (mostly in the school of S.~Sidki). In some papers, a group is called self-similar if it is generated by all the states of a finite automaton, which is again different from the standard sense since in the latter one there is no condition of finiteness. In such papers sometimes standard self-similar groups are called ``weakly self-similar''. Also, the terms of ``fractal'' and ``semifractal'' were used in the past to denote self-similar or self-replicating groups. Finally, the terms ``self-similar'' and self-replicating were also used in earlier papers to denote regular branch groups and branching subgroups~\cite{grigorch_hz:profinite_completions00}. The reader should be aware of these differences and nuances.

\section{Isometries and dilations of local fields and their rings of integers}
\label{sec:isometries}
In this section, for a (non-Archimedean) local field $F$ with a ring of integers $\mathcal O$, we identify the groups $\Isom(\mathcal O)$ of isometries of $\mathcal O$ and $\D(F)$ of dilations of $F$ with groups of automorphisms of regular trees.

\subsection{Local fields}
Recall that a discrete valuation $v$ on a commutative field $F$ is a homomorphism from the multiplicative group $F^*$ onto $\Z$ satisfying $v(a+b)\geq\inf\{v(a),v(b)\}$ for all $a,b\in F$, taking into account the convention $v(0)=\infty$. Let $\mathcal O=\{a\in F\colon v(a)\geq 0\}$ be the \emph{ring of integers} of $F$ and $\mathfrak m$ be the maximal ideal in $\mathcal O$. Then $F$ is called a \emph{(non-archimedean) local field} if
\begin{itemize}
\item[(i)] the cardinality $q$ of the \emph{residue field} $\mathcal O/\mathfrak m$ is finite,
\item[(ii)] $F$ is complete when equipped with the metric $d_v(a,b)=q^{-v(a-b)}$.
\end{itemize}

Examples of local fields include the field $\Q_p$ of $p$-adic numbers and its finite extensions and the field of formal Laurent series $F_q((x))$ over a field $F_q$ of finite characteristic $p$.

Each element $\pi\in\mathfrak m$ with $v(\pi)=1$ generates $\mathfrak m$ as an ideal in $\mathcal O$. Let $X\subset F$ be the set of representatives of the residue field such that $0\in X$. Then each element $a\in F^*$ can be uniquely written as a formal Laurent series
\begin{equation}
\label{eq:series}
a=\sum_{i=n}^\infty x_i\pi^i, x_i\in X, x_n\neq 0,
\end{equation}
where $n=v(a)\in\Z$. This identification can be used to identify any local field $F$ as a metric space with metric $d_v$ with the punctured boundary $\partial_\omega\widetilde T_{q+1}=\partial_\omega\widetilde T_{q+1}-\omega$ of the $(q+1)$-regular tree $\widetilde T_{q+1}$ for a fixed end $\omega$ of the tree, equipped with the corresponding ultrametric. Under this identification the ring of integers $\mathcal O$ is identified with the boundary of a rooted tree $T_q$ that corresponds to the series of the form~\eqref{eq:series} with $n=v(a)\geq 0$ (i.e., with no negative powers of $\pi$. We refer the reader to~\cite{serre:local_fields79,cartwright_kw:random_walks_on_the_affine_group94} for details. This tree is sometimes called Serre's tree and is a particular case of a Bruhat-Tits building~\cite{serre:trees}. The affine group $\mathrm{Aff}(F)$, which is a subgroup of $\mathrm{GL}(2,F)$, acts on $\widetilde T_{d+1}$ and is isomorphic to a closed subgroup of $\Aut(\widetilde T_{d+1})$.

\subsection{Actions on rooted and unrooted regular trees}

For the sake of simplicity of notation we will explicitly describe the correspondence~\eqref{eq:series} on the example of the field $\Q_p$. However, the same construction with minimal changes works in the general case. The ring of $p$-adic integers $\Z_p$ is the ring of integers of $\Q_p$ with respect to the $p$-adic metric $d_p$ induced by $p$-adic valuation $v_p$. The maximal ideal $\mathfrak m=p\Z_p$ in $\Z_p$ is generated by a uniformizer $p=.010^\infty$ and $X=\{0,1,\ldots,p-1\}$ is the set of representatives of $\mathfrak m$ in $\Z_p$. Then every element $a$ of $\Q_p^*$ can be uniquely written as
\begin{equation*}
a=\sum_{i=n}^\infty x_ip^i, x_i\in X, x_n\neq 0,
\end{equation*}
where $n=v_p(a)\in\Z$. This decomposition corresponds to the standard representation of elements of $\Q_p$ as biinfinite sequences over $X$ that has finite number of non-zero symbols at positions with negative indexes.

\begin{definition}
\label{def:dilation}
For a metric space $(\mathcal Y,d_{\mathcal{Y}})$ a homeomorphism $f\colon\mathcal Y\to\mathcal Y$ called a \emph{dilation} if there is $\lambda>0$ such that for all $x,y\in\mathcal Y$
\[d_{\mathcal Y}\bigl(f(x),f(y)\bigr)=\lambda d_{\mathcal Y}(x,y).\]
\end{definition}

For an alphabet $X=\{0,1,\ldots,p-1\}$, the set $X^n$ of vertices of level $n$ of $T(X)$ can be naturally identified with the ring $\Zp$ via
\[X^n\ni x_0x_1 \ldots x_{n-1} \longleftrightarrow x_0+p\cdot x_1+ \cdots +p^{n-1}\cdot x_{n-1} \in \Zp.\]
The boundary $\partial T(X)$ of $T(X)$ is then identified with the ring $\Z_p$ of $p$-adic integers:
\[\partial T(X)\ni x_0x_1x_2\ldots \longleftrightarrow \sum_{i=0}^\infty x_ip^i\in\Z_p.\]

Now for arbitrary $d\geq 2$ (not necessarily prime) we build an infinite unrooted $(d+1)$-regular tree $\widetilde T_{d+1}$ with a selected end (point of the boundary), whose punctured boundary (to be defined below) for the case of prime $d=p$, will be identified with $\Q_p$. In our notation the ``tilde'' is used to separate the rooted case from the unrooted one.

As a graph the tree $\widetilde T_{d+1}$ is an infinite regular (non-rooted) $(d+1)$-ary tree. We build it as a direct limit of a sequence of graphs $T_d^{(n)}$, $n\geq 0$ each of which is isomorphic to the rooted $d$-ary tree $T_d=T(X)$, for $X=\{0,1,\ldots,d-1\}$. Namely, we define
\[\widetilde T_{d+1}=\bigcup_{n=0}^{\infty}T_d^{(n)},\]
where embedding of $T_d^{(n)}$ into $T_d^{(n+1)}$ is induced by the map
\begin{equation}
\label{eqn:tree_embedding_Qp}
v\mapsto 0v,\quad v\in V(T_d)=X^*.
\end{equation}
and induced map on the set of edges. The reader should visualize $\widetilde T_{d+1}$ as a tree shown in Figure~\ref{fig:biinfinite_tree} (with some of the notation used in the figure explained later in the text).

If a vertex $v$ of $\widetilde T_{d+1}$ is at level $k$ of $T_d^{(n)}$ for some $k,n\geq 0$, then we say that it belongs to level $k-n$ of $\widetilde T_{d+1}$, where $k-n$ ranges over $\Z$. The children of a vertex $v$ of $\widetilde T_{d+1}$ of level $l$ are defined to be the neighbors of $v$ from level $l+1$. Each vertex in $\widetilde T_{d+1}$ has exactly $d$ children.

\begin{figure}[h]
\begin{center}
\epsfig{file=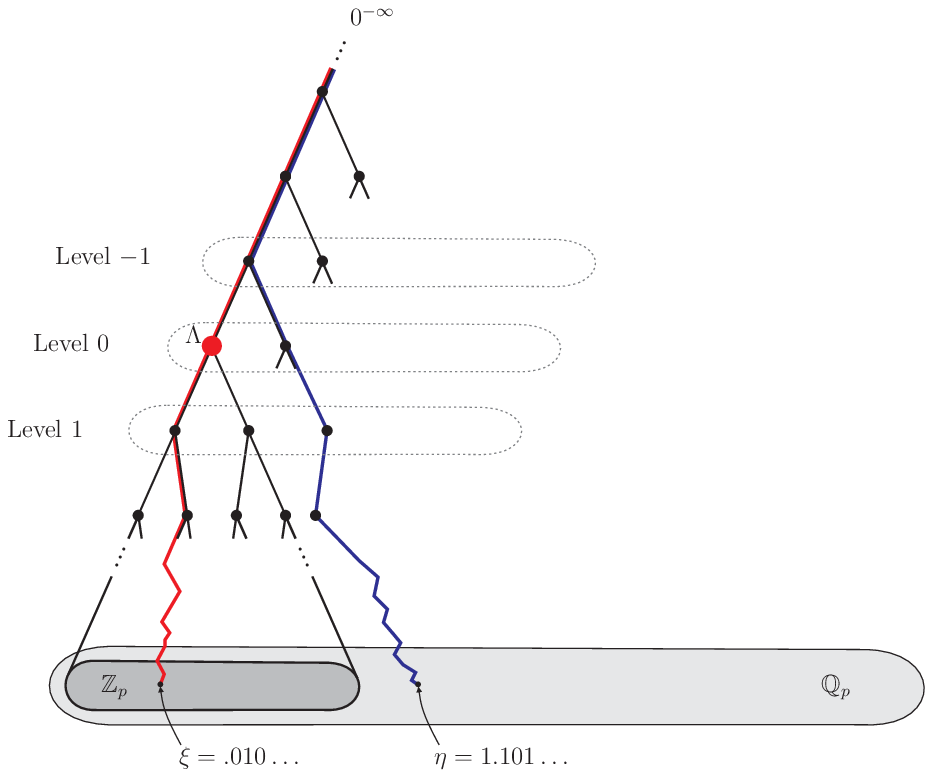}
\caption{$\Z_p$ and $\Q_p$ as boundaries of trees $T_p$ and $\widetilde T_{p+1}$\label{fig:biinfinite_tree}}
\end{center}
\end{figure}

In order to define the punctured boundary of $\widetilde T_{d+1}$ we recall that an \emph{end} of an infinite unrooted tree is an equivalence class of infinite geodesic rays (one-sided geodesics) in $\widetilde T_{d+1}$, where two rays are called equivalent if they coincide starting from some point. The collection of all ends of $\widetilde T_{d+1}$ forms the boundary $\partial \widetilde T_{d+1}$. Among all ends of the tree we select one end, denoted by $0^{-\infty}$ (or $\omega$ in most of the cases), that contains a ray consisting of the roots of rooted trees $T_d^{(n)}$, $n\geq 0$. This is the end corresponding to an infinite path going up in Figure~\ref{fig:biinfinite_tree}. All other ends of $\widetilde T_{d+1}$ correspond to rays going down as shown again in Figure~\ref{fig:biinfinite_tree}.

\begin{definition}
The punctured boundary $\partial_\omega\widetilde T_{d+1}$ of the tree $\widetilde T_{d+1}$ is the subset $\partial\widetilde T_{d+1}\setminus\{\omega\}$ obtained from $\partial \widetilde T_{d+1}$ by deleting a singleton $\omega$.
\end{definition}

Any two ends of $\widetilde T_{d+1}$ can be joined by the unique bi-infinite geodesic and conversely, any bi-infinite geodesic joins two ends of $\widetilde T_{d+1}$. Thus, the punctured boundary $\partial_\omega\widetilde T_{d+1}$ can be identified with the set of all bi-infinite geodesics in $\widetilde T_{d+1}$ initiating at $\omega$ (by analogy with the boundary of a rooted tree that consists of geodesic rays emanating from the root of the tree). Under this identification an end in $\partial_\omega\widetilde T_{d+1}$ corresponds to a bi-infinite geodesic connecting $\omega$ to it. Two points of the punctured boundary, denoted by $\xi$ and $\eta$ are shown in Figure~\ref{fig:biinfinite_tree}.

Define a metric on $\partial_{\omega}\widetilde T_{d+1}$ in a similar way as the metric was defined on the boundary of a rooted $d$-ary tree $T_d$. Namely, for any two elements $\xi_1,\xi_2$ of $\partial_{\omega}\widetilde T_{d+1}$ (viewed as bi-infinite geodesics) the distance $d_{\partial_{\omega}\widetilde T_{d+1}}(\xi_1, \xi_2)$ between them is defined as
\begin{equation}
\label{eqn:bndry_distance}
d_{\partial_{\omega}\widetilde T_{d+1}}(\xi_1, \xi_2)=d^{-l+1},
\end{equation}
where $l$ is the smallest level where $\xi_1$ differs from $\xi_2$. Under this metric $\partial_{\omega}\widetilde T_{d+1}$ becomes an ultrametric space.

For a prime $p$, we now identify $\partial_{\omega}\widetilde T_{p+1}$ with the field $\Q_p$ of $p$-adic numbers by producing an appropriate labelling of its elements. Note that a similar identification can be made for arbitrary $d\geq 2$ between $\partial_{\omega}\widetilde T_{d+1}$ with the ring $\Q_d$ of $d$-adic numbers, but we will not consider this case separately as all the technical details are identical.

First, there is a unique bi-infinite geodesic from $\partial_{\omega}\widetilde T_{d+1}$ consisting of the roots of the trees $T_d^{(n)}$, $n\geq 0$ and vertices $0^m$, $m>0$ in $T_d^{(0)}$, called the \emph{spine}. It is shown in Figure~\ref{fig:biinfinite_tree} as the path connecting $0^{-\infty}$ to $0^{\infty}$. This element will represent zero in $\Q_p$.

By definition of $\partial_{\omega}\widetilde T_{p+1}$ for each $\xi\in \partial_{\omega}\widetilde T_{p+1}$ there is $N\in\Z$ (not unique) such that $\xi$ follows the spine outside $T_p^{(N)}$. In other words, $\xi$ branches out of the spine inside $T_p^{(N)}$. We will then label each $\xi$ as a bi-infinite word $\ldots x_{-2}x_{-1}x_{0}.x_{1}x_{2}x_{3}\ldots$ over $X$, where $x_i=0$ for $i\leq -N$ and $x_i$ is the $(i+N)$-th letter of the label of the restriction of $\xi$ to $T_d^{(N)}$ for $i>-N$, and the ``dot'' symbol is placed between the letters at positions 0 and 1. In other words, $x_i$ indicates what happens with $\xi$ right below level $i$. For example, the labels of $\xi$ and $\eta$ in Figure~\ref{fig:biinfinite_tree} are $\ldots000.010\ldots$ and $\ldots001.101\ldots$ that we will usually abbreviate as $.010\ldots$ and $1.101\ldots$, respectively, by omitting the beginning 0's before the first nonzero digit or the dot symbol. The spine of $\widetilde T_{d+1}$ is then labelled by $\ldots000.000\ldots=0^{-\infty}.0^{\infty}$ or just by $.0^{\infty}$.

Under the above labelling, the distance formula~\eqref{eqn:bndry_distance} can be rewritten as
\begin{equation}
d_{\partial_{\omega}\widetilde T_{d+1}}(\ldots x_{-2}x_{-1}x_{0}.x_{1}x_{2}x_{3}\ldots,\ldots y_{-2}y_{-1}y_{0}.y_{1}y_{2}y_{3}\ldots)=d^{-\min\{i\in\Z\colon x_i\neq y_i\}+1}.
\end{equation}

Now define a natural bijection $\varphi\colon \partial_{\omega}\widetilde T_{p+1}\to\Q_p$ by
\[\varphi(\ldots x_{-3}x_{-2}x_{-1}x_{0}.x_{1}x_{2}x_{3}\ldots)=\sum_{i=-\infty}^{\infty} x_ip^i,\]
which is an isometry by the definition of a $p$-adic metric on $\Q_p$.

The subset
\[A=\{\ldots x_{-2}x_{-1}x_{0}.x_{1}x_{2}x_{3}\ldots\in \partial_{\omega}\widetilde T_{p+1}\colon x_i=0\ \text{for}\ i\leq0\}\]
of $\partial_{\omega}\widetilde T_{p+1}$ is naturally isometric under $\varphi$ to the ring $\Z_p$ of $p$-adic integers, and also to the boundary $\partial T_p^{(0)}=\partial T_p$ of the rooted $p$-ary tree supplied with the standard metric.

The identification $\varphi$ also gives rise to a natural labelling by (equivalence classes of) pairs $(\xi,n)$, $\xi\in\partial_{\omega}\widetilde T_{d+1}$ and $n\in\Z$, of the vertices of $\widetilde T_{d+1}$. Indeed, every vertex of level $n\in\Z$ can be identified (non-uniquely) with a pair $(\xi, n)$, where $\xi$ is a bi-infinite word over $X$ with the dot symbol representing a bi-infinite geodesic connecting $\omega$ to another end and passing through $v$. Clearly, if $\xi'$ coincides with $\xi$ at all positions $m\leq n$, then $(\xi,n)$ and $(\xi', n)$ correspond to the same vertex of $\widetilde T_{d+1}$. This defines an equivalence relation on the set of all such pairs $\{(\xi,n)\colon \xi\in\partial_{\omega}\widetilde T_{d+1}, n\in\Z\}$. Then the equivalence class of a pair $(\xi,n)$, denoted by $[\xi]_n$, uniquely identifies a vertex of $\widetilde T_{d+1}$. We denote the vertices $[0^{-\infty}.0^{\infty}]_n$ of the spine by $0^{n}$, where $n\in\Z$. Finally, we denote the root $0^0$ of $T^{(0)}_d$ by $\rootv$ and color it in the red color in Figures~\ref{fig:biinfinite_tree} and~\ref{fig:embedding_theta}.

The group $\Aut(\widetilde T_{p+1})$ of all automorphisms of $\widetilde T_{p+1}$ is a locally compact group, whose elements are either \textit{elliptic} (fixing a vertex of $T_{p+1}$), \textit{inversions} (fixing an edge but not a vertex of $T_{d+1}$) or \textit{hyperbolic} (the rest). For a fixed end $\omega\in\partial_{\omega}\widetilde T_{p+1}$ let $\Aut(\widetilde T_{p+1})_{\omega}$ be a subgroup of $\Aut(\widetilde T_{p+1})$ consisting of automorphisms that preserve an end $\omega$ of $\widetilde T_{p+1}$ (i.e., such automorphism $g\in\Aut(\widetilde T_{p+1})$ for which there is $N,k\in\Z$ such that $g(0^{-n})=0^{-n+k}$ for all $n\geq N$). Let also $\Aut_0(\widetilde T_{p+1})_{\omega}$ be a subgroup of $\Aut(\widetilde T_{p+1})_{\omega}$ consisting of automorphisms that eventually preserve the vertices of an end $\omega$ of $\widetilde T_{p+1}$ (i.e., automorphisms $g\in\Aut(\widetilde T_{p+1})$ for which there is $N\in\Z$ such that $g(0^{-n})=0^{-n}$ for all $n\geq N$).

The groups $\Aut(\widetilde T_{p+1})_{\omega}$ and $\Aut_0(\widetilde T_{p+1})_{\omega}$ were studied before in~\cite{cartwright_kw:random_walks_on_the_affine_group94} under the names of \emph{affine} and \emph{horocyclic} groups of the tree $\widetilde T_{d+1}$, respectively. The levels of $\widetilde T_{p+1}$ are often called horocycles. In~\cite{bartholdi_nw:horocyclic_product_of_trees08} the affine group of $\widetilde T_{d+1}$ is defined as the group consisting of Busemann isometries. It is well known that $\Aut(\widetilde T_{p+1})_{\omega}$ is amenable as a topological group (we again quote Nebbia~\cite{nebbia:amenability_and_kunze-stein_property88} or~\cite[Lemma~12.14]{woess:rw}). Also, we note the group of automorphisms of a tree $D_{\overline{d}}$, which is the induced subgraph in $\widetilde T_{d+1}$ spanned by all the vertices of non-positive levels, was studied in~\cite{bier_ls:automorphisms_of_parabolic_trees16} in the context of Sylow $p$-subgroups in the finitary symmetric group $FS_{\mathbb N}$.

The subgroup $\Aut(\widetilde T_{p+1})_{\omega}$ of $\Aut(\widetilde T_{p+1})$ consists only of elliptic and hyperbolic elements. For an end $\omega=0^{-\infty}$, one particularly important for us element of $\Aut(\widetilde T_{p+1})_{\omega}$ is a shift $\tau\in\Aut(\widetilde T_{p+1})_{\omega}$ that acts on each $v\in\widetilde T_{d+1}$ by moving it in the direction of $O^{-\infty}$ as shown in Figure~\ref{fig:shift}. In terms of the action on the boundary $\partial_{\omega}\widetilde T_{d+1}$, it moves the dot in the expression of each element $\xi\in\partial_{\omega}\widetilde T_{d+1}$, viewed as a bi-infinite string over $X$ that have $0^{-\infty}$ as a prefix, by one position to the right, that is corresponds to the left shift map $s$ on $X^\Z$:
\[\theta(t)(\ldots x_{-2}x_{-1}x_{0}.x_{1}x_{2}x_{3}\ldots)=\ldots x_{-1}x_{0}x_{1}.x_{2}x_{3}x_{4}\ldots\]

The following theorem stated in the introduction connects the groups of automorphisms of trees with the groups of isometries and dilations of local fields and their rings of integers.

\begin{theoremisoms}
Let $F$ be a (non-Archimedean) local field with the ring of integers $\mathcal O$ and the residue field of size $q$, endowed with the metric induced by its valuation. Then\\[-5mm]
\begin{itemize}
\item[(i)] the group $\Isom(\mathcal O)$ of isometries of $\mathcal O$ is isomorphic to the group $\Aut(T_q)$ of automorphisms of $T_q$.
\item[(ii)] the group $\Isom(F)$ of isometries of $F$ is isomorphic to the group $\Aut_0(\widetilde T_{q+1})_{\omega}$ consisting of automorphisms of $\widetilde T_{q+1}$ that fix a distinguished end $\omega$ pointwise (fixing the vertices in $\omega$).
\item[(iii)] the group $\D(F)$ of dilations of $F$ is isomorphic to the group $\Aut(\widetilde T_{q+1})_{\omega}$ consisting of automorphisms of $\widetilde T_{q+1}$ that fix an end $\omega$ as a point of the boundary of $\widetilde T_{q+1}$ (but possibly moving the vertices in rays corresponding to $\omega$ along those rays while preserving $\omega$ as a point of the boundary).
\end{itemize}
\end{theoremisoms}


\begin{proof}
Since $\mathcal O$ and $F$ are isometric as metric spaces to the boundaries $\partial T_q$ and $\partial_{\omega}\widetilde T_{q+1}$ of trees $T_q$ and $\widetilde T_{q+1}$, their groups of isometries are isomorphic to the groups of isometries of these boundaries. Therefore, we need to show that these groups of isometries are isomorphic to the respective automorphism groups.

For $T_q$ this is shown in Proposition~3.7 in~\cite{gns00:automata}. For the case of $\widetilde T_{q+1}$ the proof is similar. Clearly, every automorphism of $\widetilde T_{q+1}$ fixing $0^{-\infty}$ induces an isometry of $\partial_{\omega}\widetilde T_{q+1}$ as the action on the set of vertices of $\widetilde T_{q+1}$ induces the action on the boundary. Conversely, suppose $h\in\Isom(\partial_{\omega}\widetilde T_{q+1})$ is an isometry. We will extend the action of $h$ to $V(\widetilde T_{q+1})$ by defining
\[h([\xi]_n)=[h(\xi)]_n\]
for each vertex $[\xi]_n$ of $\widetilde T_{q+1}$.
Since $h$ is an isometry on $\partial_{\omega}\widetilde T_{q+1}$, this induced action on $V(\widetilde T_{q+1})$ is an automorphism preserving the end $0^{-\infty}$ pointwise.

Finally, for the third item the proof is analogous with the only difference that the dilations of $\partial_{\omega}\widetilde T_{q+1}$ will move all vertices of $\widetilde T_{q+1}$ up or down depending on the dilation coefficient. Suppose that $h\in D(F)$ is a dilation. Since the distances between any two points in $F$ are powers of $q$, the dilation coefficient also must be of the form $q^m$ for some $m\in\Z$ (that is, $d(h(\xi),h(\xi'))=q^md(\xi,\xi')$ for all $\xi,\xi'$ in $F$). Then we extend the action of $h$ to $V(\widetilde T_{q+1})$ by defining
\[h([\xi]_n)=[h(\xi)]_{n+m}\]
for each vertex $[\xi]_n$ of $\widetilde T_{q+1}$.
\end{proof}

Similarly to sections of automorphisms of $T_d$, we define the sections of automorphisms from $\Aut_0(\widetilde T_{d+1})_{\omega}$ and $\Aut(\widetilde T_{d+1})_{\omega}$. However, with the lack of the root in $\widetilde T_{d+1}$ the notation become a little more involved. Suppose $g\in \Aut(\widetilde T_{d+1})_{\omega}$ and $v=[\xi]_n$ is a vertex of $\widetilde T_{d+1}$ that is moved by $g$ to a vertex $w=[\eta]_m$. Then the rooted $d$-ary tree $T_v$ hanging down at vertex $v$ in $\widetilde T_{d+1}$ is mapped bijectively to the tree $T_w$ hanging down at vertex $w$. Since both of this trees are canonically isomorphic to $T_d$, this induces the unique automorphism $g|_v$ of $T_d$ that we will call the \emph{section of $g$ at $v$}. Formally, we define it as follows. If $u\in X^*$ is a vertex of $T_d$, then $g|_v(u)=u'$ if and only if
\[g\left(\bigl[[\xi]_nu0^\infty\bigr]_{n+|u|}\right)=\bigl[[\eta]_mu'0^\infty\bigr]_{m+|u'|},\]
where $[\xi]_nu0^\infty$ denotes the bi-infinite word obtained from $\xi$ by replacing the tail of $\xi$ starting from position $n+1$ by $u0^\infty$, and similarly for $[\eta]_mu'0^\infty$.

Similarly to the definition of a projection homomorphism~\eqref{eqn:proj}, for any vertex $v\in V(\widetilde T_{d+1})$ we define a natural homomorphism
\begin{equation}
\label{eqn:proj2}
\begin{array}{llll}\pi_v\colon &\St_{\Aut(T(X))}(v)&\to&\Aut(T(X))\\
&g&\mapsto&g|_v
\end{array}
\end{equation}
Since for each subgroup $G$ of $\Aut(\widetilde T_{d+1})_{\omega}$ we have $\Stg(v)<\St_{\Aut(T(X))}(v)$, the homomorphism $\pi_v$ naturally restricts to a homomorphism $\Stg(v)\to\Aut(T(X))$.

Also, similarly to the case of rooted trees, we can visualize the action of an automorphism $g\in\Aut_0(\widetilde T_{d+1})_{\omega}$ using the notion of the \emph{portrait} of $g$, also denoted by $\mathcal P(g)$. It is a labeled infinite unrooted $(|X|+1)$-ary tree with the vertex $\rootv$ (the root of $T^{(0)}_d$) labeled by $g$ and each vertex $v$ labeled by $g|_v$. Under each vertex $v$, we write the name of the mapping that $g|_v$ defines on the first level of the subtree hanging from $v$. In the case $d=2$, we draw an arc (called switch) connecting the two edges hanging down from $v$ if $h|_v$ acts nontrivially on the first level of the subtree hanging from $v$. If there is no switch, it means the action on the first level is trivial. Examples of portraits corresponding to the generators of the finitely presented group $\fpG$ that will be discussed below are shown in Figures~\ref{fig:theta_bcd} and~\ref{fig:theta_a}.

\subsection{Notational remarks}
\label{ssec:history_trees}

In~\cite{bartholdi_nw:horocyclic_product_of_trees08} the notation $\mathrm{Aff}(\widetilde T_{d+1})$ is suggested for $\Aut(\widetilde T_{d+1})_{\omega}$ but we think that the prefix ``$\mathrm{Aff}$'' may confuse the reader because of its connection to linearity. The group $\mathrm{Aff}(F)$ of all affine transformations of the field $F$, which is embedded in $\Aut(\widetilde T_{d+1})_{\omega}$ as a proper subgroup, is a linear group satisfying the Tits alternative, and hence the exotic groups like groups of Burnside type, of intermediate growth, or of branch type avoid it. At the same time the group $\Aut(\widetilde T_{d+1})_{\omega}$ is full of such exotic groups as explained in this article.

\section{Liftable groups and actions of their HNN extensions on regular trees}
\label{sec:embedding}

In this section we embed certain ascending HNN extensions of liftable self-similar groups into the group $\Aut(\widetilde T_{d+1})_\omega$.


\subsection{Definition and main embedding result}

\begin{definition}
\label{def:liftable}
A self-similar group $G$ acting on a tree $T_d=T(X)$ is called \emph{liftable} if there exists some $i\in X$ and a homomorphism $\sigma\colon G\to \Stg(i)$, called the \emph{lifting}, that is the right inverse of the projection map $\pi_i$ defined by~\eqref{eqn:proj} (i.e., such that $\pi_i\circ\sigma$ is the identity on $G$).
\end{definition}

Note that it follows immediately from the definition that the lifting $\sigma$ must be a monomorphism.

An equivalent definition of a liftability of a group is that for some $i\in X=\{0,1,\ldots,d-1\}$ the sequence
\[1\to \ker\pi_i\hookrightarrow\Stg(i)\stackrel{\pi_i}{\twoheadrightarrow} G\to 1\]
is a short exact sequence defining a split extension, and hence the stabilizer $\Stg(i)$ of vertex $i$ decomposes as a semidirect product:
\[\Stg(i)=\ker\pi_i\rtimes G.\]

It follows from the definition of a liftable group that the projection map $\pi_i\colon\Stg(i)\to G$ must be surjective. This is always the case when $G$ is self-replicating. However, the fact that $\pi_i$ is onto for some $i$ does not automatically imply the existence of a lifting. On the other hand we do not know of any example of a self-replicating non-liftable group.

\begin{question}
Is there a self-replicating non-liftable group?
\end{question}


We proceed to the main theorem of this section stated in the introduction. Recall that $\rootv=0^0$ denotes the leftmost vertex of level 0 in $\widetilde T_{d+1}$ that is also the root of the tree $T_d^{(0)}$.

\begin{theoremhnn}
Let $G$ be a liftable group acting on $T_d$ and $\sigma\colon G\to\Stg(i)$ be the corresponding lifting. Then
\begin{itemize}
\item[(i)] There is an embedding $\theta$ of the ascending HNN extension $\widetilde G$ of $G$ by $\sigma$
\[\widetilde G=\langle G,t\mid \text{relations in}\ G, tgt^{-1}=\sigma(g)\ \text{for all}\ g\in G\rangle\]
into $\Aut(\widetilde T_{d+1})_{\omega}$ for an end $\omega$ of $\widetilde T_{d+1}$.
\item[(ii)] If $G$ acts transitively on the first level of $T_d$, then $G$ is self-replicating, $\theta(\widetilde G)$ acts transitively on the set of vertices of $\widetilde T_{d+1}$, and the closure $W$ of $\theta(\widetilde G)$ in $\Aut(\widetilde T_{d+1})_{\omega}$ is a scale group.
\item[(iii)] The projection $\pi_{\rootv}$ maps $\St_{\theta(\widetilde G)}(\rootv)$ onto $G$ and $\St_{\theta(\widetilde G)}(\rootv)\cong\bigl(\ker\pi_{\rootv}\cap\theta(\widetilde G)\bigr)\rtimes \theta(G)$.
\item[(iv)] The projection $\pi_{\rootv}$ maps $\St_{W}(\rootv)$ onto $\overline{G}$ and $\St_{W}(\rootv)\cong\bigl(\ker\pi_{\rootv}\cap W\bigr)\rtimes \theta(\overline{G})$.
\end{itemize}
\end{theoremhnn}

\begin{proof}
By relabelling of $X$ we may assume that $X=\{0,1,\ldots,d-1\}$ for $d=|X|$ and that $i=0$. We will then set $\omega=0^{-\infty}$ and build an embedding $\theta\colon \widetilde G\to\Aut(\widetilde T_{d+1})_{\omega}$ as follows. Given $g\in G$ we define $\theta(g)\in \Aut_0(\widetilde T_{d+1})_{\omega}$ so that its section at vertex $0^{-n}$ is $\sigma^n(g)$ for $n\geq 0$, as shown in Figure~\ref{fig:embedding_theta}. This is a well-defined embedding of $G$ into $\Aut_0(\widetilde T_{d+1})_{\omega}$ (which we will denote by $\theta|_G$ for now and extend to $\theta\colon \widetilde G\to \Aut(\widetilde T_{d+1})_{\omega}$ below) as follows from the fact that $\sigma$ is a monomorphism with the image inside $\Stg(0)$ and such that $\pi_0\circ\sigma$ is the identity endomorphism on $G$.

\begin{figure}[h]
\begin{center}
\epsfig{file=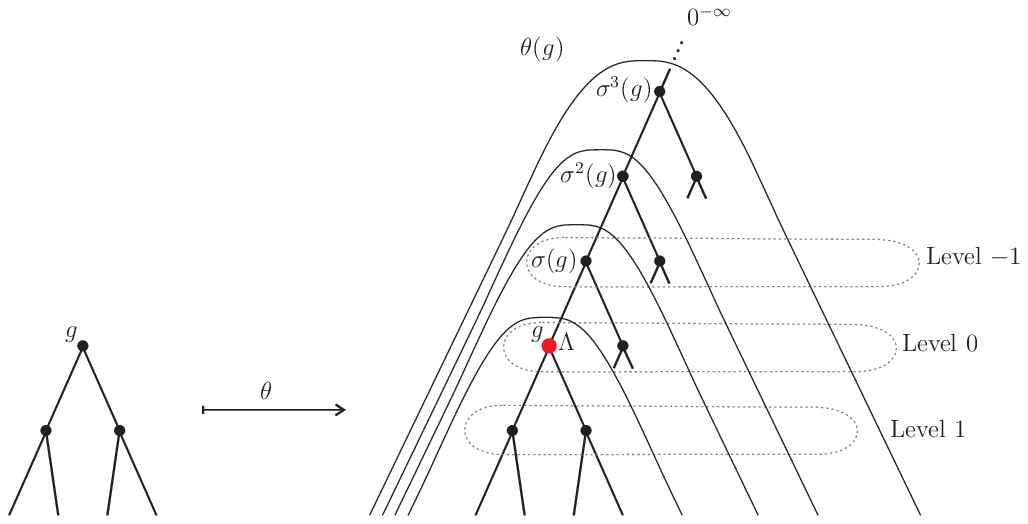,width=\textwidth}
\caption{Embedding of a self-similar group into $\Aut_0(\widetilde T_{d+1})_{\omega}$\label{fig:embedding_theta}}
\end{center}
\end{figure}

Alternatively, it can be explained in the following way. The map $\theta|_G\colon G\to\Aut_0(\widetilde T_{d+1})_{\omega}$ induces a map
\begin{equation*}
\begin{array}{llll}
\phi\colon &G&\to&\prod_{-\infty}^0 G\\[1mm]
&g&\mapsto&\bigl(\ldots,\sigma^n(g),\ldots,\sigma(g),g\bigr),
\end{array}
\end{equation*}
where the component of $\phi(g)$ at coordinate $-i$ is the section of $\theta(g)$ at vertex $0^{-i}$. Since $\theta|_G$ is a monomorphism and $\sigma(G)<\Stg(0)$, the map $\phi$ is also a monomorphism. Clearly $\theta(G)\cong\phi(G)\cong G$.

The element $\theta(t)$ is defined as a hyperbolic automorphism $\tau\in\Aut(\widetilde T_{d+1})_{\omega}$ that represents the right shift along the geodesic ray $\alpha=(0^{-\infty},0^\infty)$ defined in Section~\ref{sec:isometries}. Its action on the $\widetilde T_{p+1}$ is shown in Figure~\ref{fig:shift}.
\begin{figure}[h]
\begin{center}
\epsfig{file=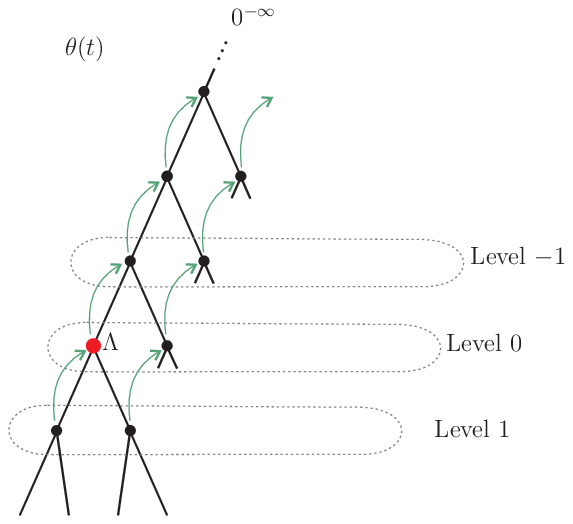}
\caption{Embedding of the stable letter of the HNN extension $\widetilde G$ into $\Aut_0(\widetilde T_{d+1})_{\omega}$\label{fig:shift}}
\end{center}
\end{figure}

We claim that so defined $\theta$ extends to a monomorphism from $\widetilde G$ to $\Aut(\widetilde T_{d+1})_{\omega}$. To justify that $\theta$ is a homomorphism it is enough to check that the family of relations $tgt^{-1}=\sigma(g)$, $g\in G$ agrees with $\theta$:
\[\theta(t)\theta(g)\theta(t)^{-1}=\theta(\sigma(g))\]
or, equivalently,
\[\phi(t)\phi(g)\phi(t)^{-1}=\phi(\sigma(g)).\]

By constructions of $\theta(g)$ we have
\[\phi(\sigma(g))=\bigl(\ldots,\sigma^{n+1}(g),\ldots,\sigma^2(g),\sigma(g)\bigr).\]
On the other hand, the portrait of $\theta(t)\theta(g)\theta(t)^{-1}$ is obtained from the portrait of $\theta(g)$ by shifting all the labels towards the end $0^{-\infty}$, thus
\[\phi(t)\phi(g)\phi(t)^{-1}=\bigl(\ldots,\sigma^{n+1}(g),\ldots,\sigma^2(g),\sigma(g)\bigr)\]
as well, so $\theta$ is a homomorphism.

Now we claim that $\theta$ is injective. Let $H=\theta(G)\cong G$ and $\eta\colon H\to\theta(\widetilde G)$ be a conjugation monomorphism
\[\eta(h)=\theta(t)h\theta(t)^{-1}.\]
We claim that the image of $\eta$ is in $H$. Indeed, if $h=\theta(g)\in H$ then
\[\eta(h)=\theta(t)\theta(g)\theta(t)^{-1}=\theta(tgt^{-1})=\theta(\sigma(g))\in H\]
since $\sigma(g)\in G$. So $\eta\colon H\to H$ is a monomorphism (since $\sigma$ and $\theta|_G$ are monomorphisms) and the diagram
\[\begin{tikzcd}
G \arrow[r, "\sigma"] \arrow[d, "\theta|_G"]& G \arrow[d,"\theta|_G"] \\ H \arrow[r, "\eta" ] & H
\end{tikzcd}
\]
is commutative.

Hence, $\widetilde H=\langle H, \theta(t)\rangle=\langle H, \theta(t)\mid \theta(t)h\theta(t)^{-1}=\eta(h), h\in H\rangle$ is an ascending HNN extension isomorphic to $\widetilde G$ and the embedding $\theta\colon\widetilde G\to\Aut(\widetilde T_{d+1})_{\omega}$ is justified, which completes the proof of item (i).

To prove the transitivity claim from item (ii), we use the well-known (and easily proved) fact that if a self-similar group $G$ satisfies $\pi_i(\Stg(i))=G$ for some vertex $i$ of the first level, and $G$ acts transitively on the first level, then $G$ is self-replicating group and it acts transitively on all levels of the tree.

Given two vertices $u,v\in V(\widetilde T_{d+1})$ we can find $i,j\in\Z$ such that for $\tau=\theta(t)$: $\tau^i(u)$ and $\tau^j(v)$ are vertices at the same level of $T_d^{(0)}$. Then, using level transitivity of $G$ (and hence of $\theta(G)$ on levels of $T_d^{(0)}$), we can map $\tau^i(u)$ to $\tau^j(v)$ by an element of $\theta(G)$.

Now the fact that the closure $\overline{\theta(\widetilde G)}$ in $\Aut(\widetilde T_{d+1})_{\omega}$ is a scale group fixing $\omega=0^{-\infty}$ is obvious. Moreover, $\overline{\theta(\widetilde G)}$ has the structure of the semidirect product $\overline{\theta(G)}\rtimes \langle\tau\rangle$.


For item (iii) we observe that $\theta(G)<\St_{\theta(\widetilde G)}(\rootv)$ and $\pi_\rootv(\theta(g))=g$ for all $g\in G$, so $G<\pi_\rootv(\St_{\theta(\widetilde G)}(\rootv))$. On the other hand, since all sections of $\theta(t)$ are trivial, all sections of elements of $\theta(\widetilde G)$ at all vertices are in $G$, therefore $\pi_\rootv(\St_{\theta(\widetilde G)}(\rootv))<G$.

To justify the semidirect product structure of $\St_{\theta(\widetilde G)}(\rootv)$ we note that since the action of $\theta(G)$ on the subtree $T_{\rootv}$ is faithful (because it coincides with the action of $G$ on $T_d$),
\begin{equation}
\label{eqn:intersection}
\theta(G)\cap\bigl(\ker\pi_{\rootv}\cap\theta(\widetilde G)\bigr)=\{1\},
\end{equation}
On the other hand,
\begin{equation}
\label{eqn:product}
\St_{\theta(\widetilde G)}(\rootv)=\bigl(\ker\pi_{\rootv}\cap\theta(\widetilde G)\bigr)\cdot \theta(G).
\end{equation}
Indeed, since all sections of the shift $\theta(t)$ are trivial, for each $s\in\St_{\theta(\widetilde G)}(\rootv)$ we have $g:=s|_\rootv\in G$, so $h=s\cdot\theta(g)^{-1}\in\ker\pi_\rootv\cap\theta(\widetilde G)$. Therefore, $s=h\cdot\theta(g)\in\bigl(\ker\pi_\rootv\cap\theta(\widetilde G)\bigr)\cdot\theta(G)$. Equations~\eqref{eqn:intersection} and~\eqref{eqn:product} yield that $\St_{\theta(\widetilde G)}(\rootv)=\ker\pi_{\rootv}\rtimes \theta(G)$.

Item (iv) is proved analogously to item (iii).

\end{proof}

\subsection{Connection to the results of Willis on scale groups}
\label{ssec:connection_to_willis}
Here we briefly discuss the relation between Theorem~\ref{thmx:hnn} and recent results of George Willis, who noticed a connection between the classes of self-replicating and scale groups~\cite{willis:scale_groups}. Willis showed in Proposition~3.8 that, given a scale group acting on $\Aut(\widetilde T_{d+1})_\omega$, the stabilizers of vertices of $\widetilde T_{d+1}$ in this group project to the same subgroup of $T_d$ that is self-replicating. Conversely, in Proposition~3.10 he proved that given a closed self-replicating group $G$ acting on $T_d$, there is a scale group $\widetilde G<\Aut(\widetilde T_{d+1})_\omega$ such that $\pi_\Lambda(\St_{\widetilde G}(\Lambda))=G$. The group $\widetilde G$ is constructed as a subgroup of $\Aut(\widetilde T_{d+1})_\omega$ consisting of all automorphisms of $\widetilde T_{d+1}$ whose sections at all vertices of that tree belong to $G$. This part of the article does not address the question of when one can embed a closed self-replicating group into a scale groups, but realizes the former as a section of the latter. Theorem~\ref{thmx:hnn}, on the other hand, provides explicit way to embed many self-replicating groups (not necessarily closed) and their HNN extensions into scale groups under condition of liftability. Moreover, it yields many examples of countable discrete transitive subgroups of $\Aut(\widetilde T_{d+1})_\omega$, whose closures are scale groups, as described in Section~\ref{sec:examples} below.

\section{Examples of liftable groups}
\label{sec:examples}

As was mentioned in the introduction, there are several main sources of groups satisfying the conditions of Theorem~\ref{thmx:hnn}.

\subsection{Liftable groups coming from essentially free actions}
The first class is the class of self-replicating groups acting \emph{essentially freely} on the boundary of $T_d$ and the second type are the groups with a suitable $L$-presentation. Recall that the action of a countable group $G$ on a probability measure space $(X,\mu)$ by measure preserving transformations is called essentially free if the measure of a set of points in $X$ that have non-trivial stabilizers is equal to 0 (or, equivalently, for each $g\in G, g\neq1$, $\mu(\mathrm{Fix}(g))=0$, where $\mathrm{Fix}(g)$ is the set of fixed points of element $g$). In our case $X=\partial T_d$ and $\mu$ is the uniform Bernoulli measure defined in Section~\ref{sec:pre}. The groups acting essentially freely on the boundaries of rooted trees have been studied in~\cite{grigorch:dynamics11eng,bartholdi_s:bsolitar,kambites-s-s:spectra,grigorch_s:essfree,nekrashevych_p:scale_invariant}. Two examples in this class are the lamplighter group $\mathcal L=\Z_2\wr\Z$ and the Baumslag-Solitar group $BS(1,3)$ (see~\cite{grigorch_s:essfree}) that are generated by automata shown in the left and right sides of Figure~\ref{fig:lamplighter} respectfully. We study them in detail below.

\begin{figure}
\begin{center}
\epsfig{file=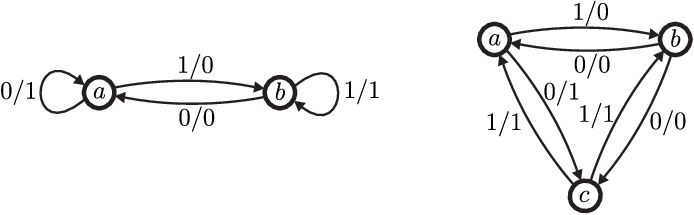}
\end{center}
\caption{The automata generating the lamplighter group (left) and Baumslag-Solitar group $BS(1,3)$ (right)\label{fig:lamplighter}}
\end{figure}

\begin{proposition}
\label{prop:ess_free_liftable}
If a group $G<\Aut(T_d)$ acts on $\partial T_d$ essentially freely and is self-replicating, then it is liftable.
\end{proposition}

\begin{proof}
We identify $T_d$ with $T(X)$ for an alphabet $X=\{0,1,\ldots,d-1\}$. Since $G$ acts essentially freely, the projection homomorphism $\pi_0\colon\Stg(0)\to G$ for a letter $0\in X$ must be injective. Indeed, if there is a $g\in\Stg(0)$, $g\neq 1$, such that $g|_0$ is trivial, then this $g$ would be in the stabilizers of all points from $\partial T_d$ that start with 0. Clearly, the set of such points has a non-zero measure, thus contradicting to the fact that the action of $G$ is essentially free. Therefore, since $G$ is self-replicating and, hence, $\pi_0(\Stg(0))=G$, we can uniquely define the right inverse $\sigma$ of the projection $\pi_0$ that will be a monomorphism as required by Theorem~\ref{thmx:hnn}.
\end{proof}

\begin{corollary}
\label{cor:ess_free}
Suppose a group $G<\Aut(T_d)$ acts on $\partial T_d$ essentially freely and is self-replicating. Then the ascending HNN extension of $G$ given by the lifting from Proposition~\ref{prop:ess_free_liftable} embeds into the group $\Aut(\widetilde T_{d+1})_{\omega}$, $\omega=0^{-\infty}$ with the image acting transitively on the set of vertices of $\widetilde T_{d+1})$.
\end{corollary}

There is an interesting connection here to the class of scale-invariant groups. We recall the definition below.
\begin{definition}
\label{def:scale_invariant}
A finitely generated group $G$ is called \emph{B-scale-invariant} if there exists a nested sequence of subgroups of finite index in $G$, each isomorphic to $G$, and whose intersection is a finite group. If the latter intersection is trivial, then $G$ is called \emph{scale-invariant}.
\end{definition}

This class was introduced by Benjamini (this is why we add ``B'' in front of ``scale-invariant'') who conjectured that every such group is virtually nilpotent. A counterexample to this conjecture, the lamplighter group $\mathcal L$ discussed in Examples~\ref{ex:lamplighter} and~\ref{ex:lamplighter2}, was provided implicitly in~\cite{grigorch_z:lamplighter} (the paper was printed before the conjecture was stated) and explicitly in~\cite{nekrashevych_p:scale_invariant}, where many other examples where produced. It was shown in~\cite[Proposition~3.13]{grigorch_s:essfree} that every self-replicating group acting essentially freely on $\partial T_d$ is scale-invariant. That proposition gave a new place to search for scale-invariant groups. Therefore, by Proposition~\ref{prop:ess_free_liftable}, there are examples of liftable scale-invariant groups, some of which are given below.

\begin{example}
\label{ex:lamplighter}
Lamplighter group $\mathcal L\cong\Z_2\wr\Z$ is a 2-generated group that can be realized as a self-similar group defined by the following wreath recursion corresponding to the automaton representation given by automaton on the left side of Figure~\ref{fig:lamplighter}~\cite{gns00:automata,grigorch_z:lamplighter}:
\[\begin{array}{rcl}
a&=&(a,b)\varepsilon\\
b&=&(a,b)
\end{array}
\]

The self-replicating action of $\mathcal L$ on $T_2$ has been studied in~\cite{grigorch_z:lamplighter,nekrashevych_p:scale_invariant,grigorch_k:lamplighter}. In particular, it was shown in~\cite{grigorch_s:essfree} that this group acts essentially freely on $\partial T_2$ so $\mathcal L$ satisfies the conditions of Proposition~\ref{prop:ess_free_liftable} and there exist monomorphisms $\sigma_0\colon\mathcal L\to\mathop{\rm St}\nolimits_{\mathcal L}(0)$ and $\sigma_1\colon\mathcal L\to\mathop{\rm St}\nolimits_{\mathcal L}(1)$ given by:
\begin{equation}
\label{eqn:sigma1_2_lamplighter}
\sigma_0\colon\left\{
\begin{array}{l}
a\mapsto b\\
b\mapsto b^a
\end{array}\right.\qquad\qquad
\sigma_1\colon\left\{
\begin{array}{l}
a\mapsto b^a\\
b\mapsto b
\end{array}\right.
\end{equation}
with $\pi_i\sigma_i=id_{\mathcal L}$, $i=0,1$, because $b=(a,b)$ and $b^a=(b,a)$. In other words, both $\sigma_0$ and $\sigma_1$ are liftings of $\mathcal L$. Hence, we can apply Theorem~\ref{thmx:hnn} and embed corresponding HNN extensions in $\Aut(\widetilde T_3)_\omega\cong\D(\Q_2)$.

The lamplighter group will appear again in the last section of the paper.


\end{example}

\begin{example}
The Baumslag-Solitar group $BS(1,3)=\langle s,r\mid rsr^{-1}=s^3\rangle$ can be realized as a 3-generated self-similar group defined by the following wreath recursion corresponding to the automaton representation given by automaton on the right side of Figure~\ref{fig:lamplighter} (see, for example,~\cite{bartholdi_s:bsolitar,bondarenko_gkmnss:full_clas32_short}):
\[\begin{array}{rcl}
a&=&(c,b)\varepsilon\\
b&=&(a,c)\\
c&=&(b,a)
\end{array}
\]

With respect to this generating set it has the following presentation:
\[BS(1,3)\cong\langle a,b,c\mid c=ab^{-1}a,bab^{-1}=ab^{-1}ab^{-1}a\rangle,\]
where the isomorphism between the two presentations given above is induced by $s\mapsto ab^{-1}$ and $r\mapsto b$ (see~\cite{bondarenko_gkmnss:full_clas32_short}, where the automaton groups act on the left and thus all words in generators have to be reversed to align with the notation of the present paper).

It acts essentially freely on $\partial T_2$~\cite{grigorch_s:essfree}, is self-replicating, and satisfies the conditions of Proposition~\ref{prop:ess_free_liftable}. The monomorphism $\sigma$ in this case is
\[\sigma\colon\left\{
\begin{array}{l}
a\mapsto b\\
b\mapsto c\\
c\mapsto bc^{-1}b
\end{array}\right..
\]
since $b=(a,c)$, $c=(b,a)$, and $bc^{-1}b=(ab^{-1}a,ca^{-1}c)=(c,ca^{-1}c)$.

\end{example}

We note that there are self-similar groups acting on $\partial T_d$ essentially freely that are not self-replicating. These examples include the free group $F_3$ and the free product $(\Z/2\Z)*(\Z/2\Z)*(\Z/2\Z)$~\cite{grigorch_s:essfree,bondarenko_gkmnss:full_clas32_short}. Since the stabilizer of any vertex projects onto a proper subgroup of $G$, so no lifting could exist. We finish this section with the following question.

\begin{question}
Are there non-amenable liftable groups?
\end{question}


\subsection{Liftable groups admitting finite $L$-presentations}

Now let us switch to groups admitting $L$-presentations. We refer the reader to~\cite{bartholdi:epimorphic03} for more details on this type of presentations by generators and relators. We give here somewhat restricted definition compared to the one given in~\cite{bartholdi:epimorphic03} that will be useful in our context.

\begin{definition}
A group $G$ admits an \emph{$L$-presentation} if there is a generating set $S$ for $G$, subsets $Q$ and $R$ of a free group $F_S$ on $S$ and an endomorphism $\phi\colon F_S\to F_S$ such that
\begin{equation}
\label{eq:lpres}
G\cong\left\langle S\mid Q, \phi^n(R), n\geq 0\right\rangle.
\end{equation}
\end{definition}

An $L$-presentation~\eqref{eq:lpres} is called \emph{finite} if $S,Q,R$ are finite and \emph{ascending} if $Q$ is empty. In the case of ascending $L$-presentation the endomorphism $\phi$ induces the endomorphism $\bar\phi\colon G\to G$. We call $L$-presentation~\eqref{eq:lpres} \emph{injective} if it is ascending and $\bar\phi$ is injective. In this case we can associate with $G$ an ascending HNN extension
\[\widetilde G=G*_{\bar\phi}=\langle S,t\mid R, s^t=\phi(s), s\in S\rangle.\]
If a self-similar group $G$ acting on $T(X)$ admits an injective $L$-presentation
\begin{equation}
\label{eq:lpres_ascending}
G\cong\left\langle S\mid \phi^n(R), n\geq 0\right\rangle,
\end{equation}
such that $\mathrm{Im}\bar\phi\subset\St_G(i)$ for some $i\in X$ and $\bar\phi$ is a lifting according to Definition~\ref{def:liftable}, we call presentation~\eqref{eq:lpres_ascending} \emph{liftable}. Observe that if~\eqref{eq:lpres_ascending} is liftable and finite then $\widetilde G$ is a finitely presented group.

%

As the main example of this type we now explicitly build an embedding of the finitely presented group $\fpG$ from~\cite{grigorch:example} into $\D(\mathbb Q_2)$ and study in detail some of its properties.

\begin{example} (Finitely presented amenable but not elementary amenable group $\fpG$).
\label{ex:grigorchuk}

\noindent Recall that the group $\G=\langle a,b,c,d\rangle$ was defined in~\cite{grigorch:burnside} as an example of an infinite finitely generated torsion group. Later it was shown to be the first example of a group of intermediate growth in~\cite{grigorch:degrees} and that it has many other unusual properties~\cite{grigorch:solved} like being branch and just-infinite. It acts on the rooted binary tree $T_2$ and can be generated by a Mealy automaton shown in Figure~\ref{fig:grig_aut} or, equivalently, by the following wreath recursion:
\begin{equation}
\label{eqn:wreath_grigorchuk}
\begin{array}{rcl}
a&=&(1,1)\varepsilon,\\
b&=&(a,c),\\
c&=&(a,d),\\
d&=&(1,b),
\end{array}
\end{equation}
where $\varepsilon$ denotes the nontrivial permutation of $X=\{0,1\}$.

The group $\G$ has a finite $L$-presentation~\eqref{eqn:presentationLys}, which formally is not ascending, but has an advantage of being minimal, i.e., no relators from~\eqref{eqn:presentationLys} can be dropped without changing $\G$~\cite{gr99:schur}. A slight modification of~\eqref{eqn:presentationLys} to
\begin{equation}
\label{eqn:presentationAsc}
\G=\langle a,b,c,d\mid \sigma^i(a^2), \sigma^i(b^2), \sigma^i(c^2), \sigma^i(d^2), \sigma^i(bcd), \sigma^i((ad)^4),\sigma^i((adacac)^4), i\geq0\rangle
\end{equation}
allows to associate with $\G$ a finite ascending $L$-presentation, which is liftable with the lifting $\sigma\colon\G\to\Stg(1)$ induced by substitution~\eqref{eqn:sigma_grigorchuk} as $\sigma$ is the right inverse of the projection map $\pi_1$. This follows from the wreath recursion~\eqref{eqn:wreath_grigorchuk} and a relation $aca=(d,a)$. Then the corresponding HNN extension $\fpG$ is a finitely presented group given by presentation~\eqref{eqn:presentation} (or by a more sophisticated but shorter presentation~\eqref{eqn:presentation_short}). By Theorem~\ref{thmx:hnn} the group $\fpG$ embeds into $\Aut(\widetilde T_{3})_{\omega}$ as a vertex transitive group. Below we will describe in detail the images under the embedding $\theta$ of the generators $a,b,c,d$ of $\G$ in $\Aut_0(\widetilde T_{3})$. Since $\sigma$ is the right inverse of $\pi_1$, and not of $\pi_0$, we will construct the action on the tree isomorphic to $\widetilde T_{3}$ for which the end $1^{-\infty}$ is fixed (so the tree ``grows up to the left'' as shown in Figure~\ref{fig:theta_bcd}). We will also denote this tree $\widetilde T_{3}$, but one has to keep in mind the difference. With this convention, in particular, $\theta(t)$ will move vertices of $\widetilde T_{3}$ along the end $1^{-\infty}$ up and to the left.

\begin{figure}
\begin{center}
\epsfig{file=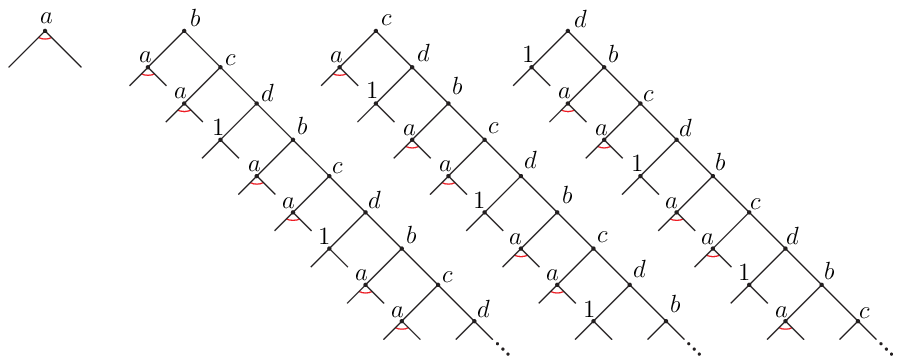}
\end{center}
\caption{Portraits of the generators of the Grigorchuk group\label{fig:grig_gens}}
\end{figure}

The action of the generators $a,b,c,d$ of $\G$ on $T_2$ is given by the portraits shown in Figure~\ref{fig:grig_gens}. Recall that portraits for elements of $\Aut(T_d)$ were defined in Section~\ref{sec:pre} and for elements of $\Aut(\widetilde T_{d+1})_{\omega}$ in Section~\ref{sec:isometries}. For the generators $b,c,d$ the portraits are 3-periodic. Hence for $\sigma$ we get the following:
\[\sigma^n(b)=\left\{
\begin{array}{l}
b,\ n\equiv0\ (\mathrm{mod}\ 3),\\
d,\ n\equiv1\ (\mathrm{mod}\ 3),\\
c,\ n\equiv2\ (\mathrm{mod}\ 3),
\end{array}\right.
\sigma^n(c)=\left\{
\begin{array}{l}
c,\ n\equiv0\ (\mathrm{mod}\ 3),\\
b,\ n\equiv1\ (\mathrm{mod}\ 3),\\
d,\ n\equiv2\ (\mathrm{mod}\ 3),
\end{array}\right.
\sigma^n(d)=\left\{
\begin{array}{l}
d,\ n\equiv0\ (\mathrm{mod}\ 3),\\
c,\ n\equiv1\ (\mathrm{mod}\ 3),\\
b,\ n\equiv2\ (\mathrm{mod}\ 3).
\end{array}\right.\]

Therefore, the portraits of $\theta(b)$, $\theta(c)$, and $\theta(c)$ on $\widetilde T_{3}$ will be periodic extensions upwards of the portraits of $b$, $c$, and $d$ on $T_2$, as shown in Figure~\ref{fig:theta_bcd}.

\begin{figure}[h]
\begin{center}
\epsfig{file=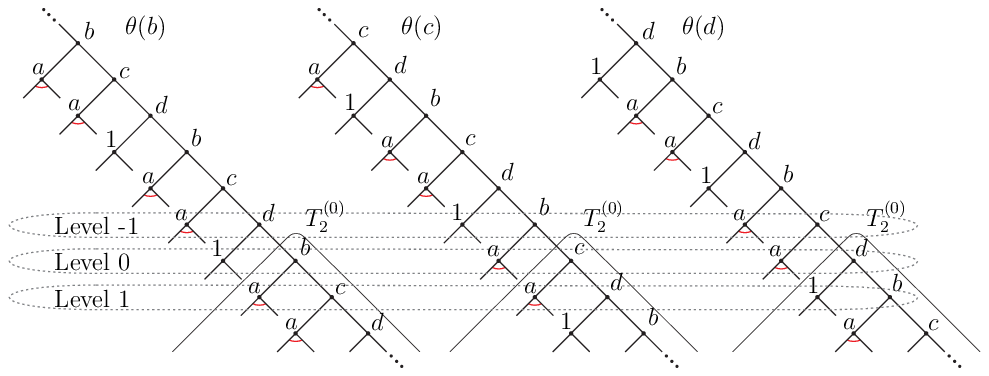}
\end{center}
\caption{Portraits of the $\theta(b)$, $\theta(c)$, and $\theta(d)$ on $\widetilde T_{3}$\label{fig:theta_bcd}}
\end{figure}

To describe $\theta(a)$ we calculate the iterations $\sigma^n(a)$.

\begin{proposition}[\cite{grigorch_ln:spectra_schreier_schroedinger18}]
\label{prop:P_n}
For each $n\geq 0$ we have $\sigma^n(a)=P_{n}$, where the sequence $\{P_n\}_{n\geq0}$ consists of palindromic words defined recursively as follows:
\[P_0=a,\quad P_{n+1}=P_nq_nP_n, n\geq 1,\]
and
\[q_n=\left\{
\begin{array}{l}
c,\ n\equiv0\ (\mathrm{mod}\ 3),\\
b,\ n\equiv1\ (\mathrm{mod}\ 3),\\
d,\ n\equiv2\ (\mathrm{mod}\ 3).
\end{array}\right.\]
\end{proposition}

\begin{proof}
Follows by induction on $n$. Clearly the base cases $n=0$ and $n=1$ are trivially satisfied as $\sigma(a)=aca=P_0q_0P_0$. Now assume that $n>1$ and $\sigma^i(a)=P_{i}$ for all $i\leq n$. Then
\[P_{n+1}=\sigma^n(a)=\sigma(P_n)=\sigma(P_{n-1}q_{n-1}P_{n-1})=\sigma(P_{n-1})\sigma(q_{n-1})\sigma(P_{n-1})=P_nq_nP_n,\]
where the last equality follows from the induction assumption and the definition of $q_n$.
\end{proof}

Observe that the sequence of palindromes $P_n$ play an important role in the study of spectra of Schreier graphs associated with $\G$~\cite{grigorch_ln:spectra_schreier_schroedinger18}.

We complete the description of the portrait of $\theta(a)$ by describing the sections of $\sigma^n(a)$ at vertex 0 of the first level of the rooted binary tree $T_2$ for all $n\geq 1$. Note, that these are also the sections of $\theta(a)$ at vertices $1^{-n}0$ of the unrooted ternary tree $\widetilde T_3$. For $n=1$ we have
\[\sigma(a)|_0=(aca)|_0=d.\]
Define an endomorphisms $\alpha_n\colon \G\to \G, n\geq 1$ as $\alpha_n=\pi_0\circ\sigma^n$, that is, for each $g\in\G$ $\alpha_n(g)=\sigma^n(g)|_0$. Visually, applying $\sigma^n$ to an element $g\in\G$ corresponds to going $n$ steps in the North-West direction in the portrait of $\theta(g)$ (along the rightmost bi-infinite path) from the vertex $\rootv$ labelled by $g$. Applying $\alpha_n$ corresponds to going $n$ steps in the North-West direction, and then one step down in the South-West direction. So in the notation of Proposition~\ref{prop:P_n} we have
\begin{equation}
\label{eqn:alpha}
\alpha_n(a)=\sigma^n(a)|_0=\sigma(\sigma^{n-1}(a))|_0=\sigma(P_{n-1})|_0=\alpha_1(P_{n-1}).
\end{equation}

\begin{proposition}
For each $n\geq 1$ we have
\[\alpha_n(a)=\left\{
\begin{array}{l}
d,\ n=1,\\
dad,\ n=2,\\
1,\ n\equiv0\ (\mathrm{mod}\ 3), n\geq 3\\
a,\ n\equiv1\ (\mathrm{mod}\ 3)\ \text{or}\ n\equiv2\ (\mathrm{mod}\ 3), n\geq 3.
\end{array}\right.\]
\end{proposition}

\begin{proof}
Induction on $n$. For $n\in\{1,2,3\}$ we have
\[\alpha_1(a)=\sigma(a)|_0=(aca)|_0=d,\]
\[\alpha_2(a)=\sigma^2(a)|_0=(aca\cdot b\cdot aca)|_0=dad,\]
\[\alpha_3(a)=\sigma^3(a)|_0=(acabaca\cdot d\cdot acabaca)|_0=dad\cdot 1\cdot dad=1,\]
Now assume that the statement is true for some $n\geq 3$. Then we calculate, using equality~\eqref{eqn:alpha}:
\[\alpha_{n+1}(a)=\alpha_1(P_{n})=\alpha_1(P_{n-1}q_{n-1}P_{n-1})=\alpha_1(P_{n-1})\alpha_1(q_{n-1})\alpha_1(P_{n-1})=\alpha_n(a)\alpha_1(q_{n-1})\alpha_n(a).\]
Now, considering three cases for $n$ $\mathrm{mod}\ 3$, we use the inductive assumption to find the values of $\alpha_n(a)\in\{1,a\}$ and Proposition~\ref{prop:P_n} to find $q_{n-1}\in\{b,c,d\}$. Finally, evaluating $\alpha_1(b)=1$, $\alpha_1(c)=a$, $\alpha_1(d)=a$ and using the fact that both $a$ and $d$ are involutions completes the proof.
\end{proof}

Understanding of $\alpha_n(a)$ helps us to describe the portrait of $\theta(a)$. Namely, the portrait of $\theta(a)$ on $\widetilde T_{3}$ is shown in Figure~\ref{fig:theta_a}. The original binary rooted tree on which the group $\G$ acts is marked in the end of the right branch by an arc. As desired by construction, $\theta(a)$ acts as $a$ on that subtree.

\begin{figure}[h]
\begin{center}
\epsfig{file=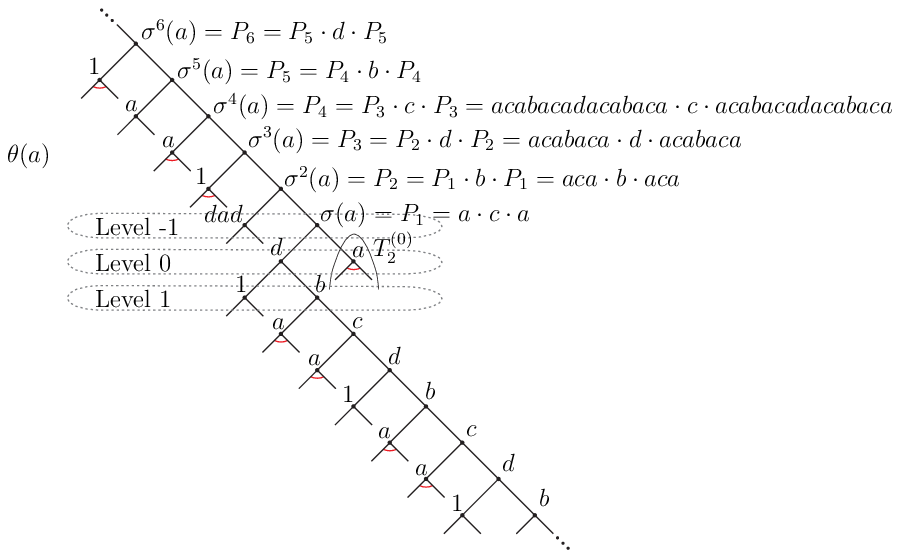}
\end{center}
\caption{Portrait of the $\theta(a)$ on $\widetilde T_{3}$\label{fig:theta_a}}
\end{figure}

\end{example}

\begin{example} (Finitely presented extension of $\IMG(z^2+i)$).\\
Another example of the branch type group for which Theorem~\ref{thmx:hnn} is applicable is the iterated monodromy group $\Gamma=\IMG(z^2+i)$ (detailed exposition of iterated monodromy groups can be found in~\cite{nekrash:self-similar,bartholdi_gn:fractal}). This group was studied by K.-U.~Bux and R.~Perez in~\cite{bux_p:iter_monodromy}, who proved that $\Gamma$ (as well as $\G$) has intermediate growth, and Z.~\v Sun\'ic and the authors studied in~\cite{grigorch_ss:img} its spectral properties. It is generated by a 4-state automaton acting on a binary tree shown in Figure~\ref{fig:z2i}. The corresponding wreath recursion defining $\Gamma$ is
\begin{equation}\label{eqn_defin_z2}
a=(1,1)\varepsilon,\quad
b=(a,c),\quad
c=(b,1),
\end{equation}
and $\Gamma$ has the following finite ascending $L$-presentation~\cite{grigorch_ss:img}:
\begin{multline}\label{eqn_pres}
\Gamma\cong\bigl\langle a,b,c\bigm| \sigma^n(a^2),\sigma^n\bigl((ac)^4\bigr),
\sigma^n([c,ab]^2),
\sigma^n([c,bab]^2),\\ \sigma^n([c,ababa]^2), \sigma^n([c,ababab]^2),
\sigma^n([c,bababab]^2), n\geq 0\bigr\rangle,
\end{multline}
where $\sigma$ is the substitution defined by
\[\sigma\colon
\begin{cases}
a\to b,\\
b\to c,\\
c\to aba.\\
\end{cases}
\]
It extends to a monomorphism $\sigma\colon \Gamma\to\St_{\Gamma}(1)$ which is a lifting for $\Gamma$: wreath recursion~\eqref{eqn_defin_z2} and the relation $aba=(c,a)$ make $\sigma$ a right inverse of $\pi_0$. This makes the presentation~\eqref{eqn_pres} liftable. Thus, by Theorem~\ref{thmx:hnn} the HNN extension $\widetilde\Gamma$ of $\Gamma$ via the homomorphism $\sigma$ is a finitely presented group with presentation
\begin{multline*}
\widetilde \Gamma=\langle a,b,c,t\mid a^2, (ac)^4, [c,ab]^2,
[c,bab]^2, [c,ababa]^2, [c,ababab]^2, [c,bababab]^2,\\
tat^{-1}=b, tbt^{-1}=c, tct^{-1}=aba\rangle,
\end{multline*}
that embeds into $\Aut(\widetilde T_3)_\omega$. The group $\widetilde\Gamma$ is another example of a finitely presented amenable but not elementary amenable group and it also has exponential growth (as well as $\fpG$ has).

\begin{figure}[h]
\begin{center}
\epsfig{file=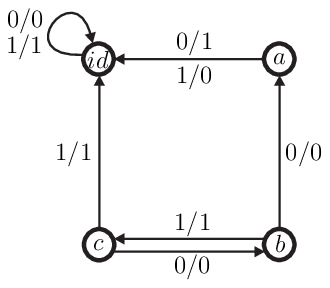,width=150pt}
\caption{Automaton generating the group $\Gamma=\IMG(z^2+i)$\label{fig:z2i}}
\end{center}
\end{figure}
\end{example}

\begin{example}(Finitely presented amenable but not subexponentially amenable group $\fpB$).
\label{ex:basilica}
Our final example in this section is the Basilica group $\mathcal B$ first studied in~\cite{grigorch_z:basilica,grigorch_z:basilica_sp}. It was the first example of a finitely generated amenable group that is not subexponentially amenable. The group $\mathcal B$ is generated by a 3-state automaton acting on a binary tree and defined by the following wreath recursion:
\begin{equation}\label{eqn_defin}
a=(b,1)\varepsilon,\quad
b=(a,1).\quad
\end{equation}
It was shown in~\cite[Proposition~2]{grigorch_z:basilica_sp} that $\mathcal B$ has the following presentation:
\begin{equation}\label{eqn_pres_basilica}
\mathcal B\cong\langle a,b\mid \sigma_{\mathcal B}^{i}\big([b,b^{a}]\big), i=0,1,\ldots\rangle,
\end{equation}
where $\sigma_{\mathcal B}$ is the substitution defined by
\begin{equation}
\label{eqn:sigma_B}
\sigma_{\mathcal B}\colon\left\{
\begin{array}{l}
a\mapsto b\\
b\mapsto a^2
\end{array}\right.
\end{equation}
By Proposition~10 in~\cite{grigorch_z:basilica} the map $\sigma_{\mathcal B}$ extends to a monomorphism $\sigma_{\mathcal B}\colon \mathcal B\to\St_{\mathcal B}(0)$ and since $b=(a,1), a^2=(b,b)$, this monomorphism is a lifting for $\mathcal B$. Thus, by Theorem~\ref{thmx:hnn} the HNN extension $\widetilde{\mathcal B}$ of $\mathcal B$ via the homomorphism $\sigma$ embeds into $\D(\Q_2)$. The group $\widetilde{\mathcal B}$ is an example of a finitely presented amenable but not subexponentially amenable group of exponential growth.

\end{example}

\subsection{Transitivity of the scale group coming from liftable groups}
\label{ssec:grig_closure}

We now study the closures of $\theta(\fpG)$ and $\theta(\widetilde{\mathcal B})$ from Examples~\ref{ex:grigorchuk} and~\ref{ex:basilica} in $\Aut(\widetilde T_3)_\omega$. The next result completes the proof of Theorem~\ref{thmx:grig} from the Introduction. In this section $\sigma$ and $\sigma_{\mathcal B}$ denote the substitutions defined in~\eqref{eqn:sigma_grigorchuk} and~\eqref{eqn:sigma_B}, respectively.


\begin{theorem}
\label{thm:grig_2_transitivity}
The groups $\overline{\theta(\fpG)}$ and $\overline{\theta(\widetilde{\B})}$ are scale groups that act 2-transitively on the punctured boundary $\partial_{\omega}\widetilde T_{3}$ of $\widetilde T_{3}$.
\end{theorem}

The proof of this theorem is based on the $2$-transitivity property of the actions of $\G$ on the levels of $T_2$ that was discovered in~\cite{bartholdi_g:parabolic} (see also Appendix by Bekka, Grigorchuk, de la Harpe in~\cite{bekka_h:irreducibility_of_unitary03}), and a similar result for the Basilica group $\B$.

\begin{proof}
Since $\G$ and $\B$ satisfy the conditions of Theorem~\ref{thmx:hnn} and act transitively on the first level, we immediately obtain $\theta(\fpG)$ and $\theta(\fpB)$ act transitively on the set of vertices of $T_3$ and that their closures $\overline{\theta(\fpG)}$ and $\overline{\theta(\widetilde{\B})}$ are scale groups acting transitively on the punctured boundary $\partial_{\omega}\widetilde T_{3}$ of $\widetilde T_{3}$ by Proposition~2.4 in~\cite{willis:scale_groups}.

We start from proving $2$-transitivity of the action of $\overline{\theta(\fpG)}$ on $\partial_{\omega}\widetilde T_{3}$. Suppose $(\xi_1,\xi_2)$ and $(\xi'_1,\xi'_2)$ are two pairs of distinct points of $\partial_{\omega}\widetilde T_{3}$. Let $n$ (resp., $n'$) be the largest level at which $\xi_1,\xi_2$ (resp., $\xi'_1,\xi'_2$) agree. Then
elements $\theta(t)^{-n}$ and $\theta(t)^{-n'}$ of $\theta(\fpG)$ will move them to pairs $(\theta(t)^{-n}(\xi_1),\theta(t)^{-n}(\xi_2))$ and $(\theta(t)^{-n'}(\xi'_1),\theta(t)^{-n'}(\xi'_2))$ such that all of the 4 involved elements of $\partial_{\omega}\widetilde T_{3}$ pass through the root $\rootv$ of $T_2^{(0)}$ and such that $\theta(t)^{-n}(\xi_1)$ and $\theta(t)^{-n}(\xi_2)$ (resp., $\theta(t)^{-n'}(\xi'_1)$ and $\theta(t)^{-n'}(\xi'_2)$) do not agree on level 1.

Since $\G$ (and thus $\theta(\G)<\theta(\fpG)$) acts transitively on the levels of $T^{(0)}_2$, for each level $l>0$ there are elements $g_l,g'_l\in\G$ such that
\begin{equation}
\label{eqn:gl}
[\theta(g_l)\bigl(\theta(t)^{-n}(\xi_1)\bigr)]_l=[1^{-\infty}.1^{\infty}]_l
\end{equation}
and
\begin{equation}
\label{eqn:glprime}
[\theta(g'_l)\bigl(\theta(t)^{-n'}(\xi'_1)\bigr)]_l=[1^{-\infty}.1^{\infty}]_l
\end{equation}

Automorphisms $\theta(g_l)$ and $\theta(g_l')$ preserve levels of $\widetilde T_{3}$. Since $\theta(t)^{-n}(\xi_1)$ and $\theta(t)^{-n}(\xi_2)$ (resp., $\theta(t)^{-n'}(\xi'_1)$ and $\theta(t)^{-n'}(\xi'_2)$) do not agree on level 1, their images under $\theta(g_l)$ and $\theta(g_l')$ also must not agree on level 1. Therefore, by~\eqref{eqn:gl} and~\eqref{eqn:glprime} we have
\[[\theta(g_l)\bigl(\theta(t)^{-n}(\xi_2)\bigr)]_l=[1^{-\infty}.0w]_l\]
and
\[[\theta(g'_l)\bigl(\theta(t)^{-n'}(\xi'_2)\bigr)]_l=[1^{-\infty}.0w']_l\]
for some infinite words $w,w'$.

Now we use the second part of Lemma 9.10 in~\cite{bartholdi_g:parabolic}, which states, in particular, that the stabilizer $P_l$ of the vertex $1^l$ in $\G$ has $0X^{l-1}$ as one of the orbits. It implies that there is $p_l\in P_l<\G$ such that
\[[\theta(p_l)\left(\theta(g_l)\bigl(\theta(t)^{-n}(\xi_2)\bigr)\right)]_l=[\theta(g'_l)\bigl(\theta(t)^{-n'}(\xi'_2)\bigr)]_l.\]

We can rewrite the last equality as
\[[\theta((g'_l)^{-1}p_lg_l)\bigl(\theta(t)^{-n}(\xi_2)\bigr)]_l=[\bigl(\theta(t)^{-n'}(\xi'_2)\bigr)]_l.\]
Since the sequence of elements $\left(\theta((g'_l)^{-1}p_lg_l)\right)_{l\geq1}$ belongs to the compact stabilizer of the root $\rootv$ of $T^{(0)}_2$ in $\Aut_0(\widetilde T_{3})_{\omega}$, it must contain a convergent subsequence $\left(\theta((g'_{l_j})^{-1}p_{l_j}g_{l_j})\right)_{j\geq1}$ such that
\[\lim_{j\to\infty}\theta((g'_{l_j})^{-1}p_{l_j}g_{l_j})=g\in\overline{\theta(\fpG)}.\]
By construction we have
\[g(\theta(t)^{-n}(\xi_1))=\theta(t)^{-n'}(\xi'_1)\]
and
\[g(\theta(t)^{-n}(\xi_2))=\theta(t)^{-n'}(\xi'_2).\]
The last two equalities imply that the element $\theta(t)^{n'}g\theta(t)^{-n}\in\overline{\theta(\fpG)}$ moves $(\xi_1,\xi_2)$ to $(\xi'_1,\xi'_2)$.

The proof that $\overline{\theta(\widetilde{\B})}$ acts 2-transitively on $\partial_{\omega}\widetilde T_{3}$ is similar to the above proof for $\overline{\theta(\fpG)}$. We only need to establish the analog of Lemma~9.10 from~\cite{bartholdi_g:parabolic} for Basilica group. Namely, we will prove that $\St_{\B}(1^l)$ acts transitively on $0X^{l-1}$ for each $l\geq 1$. It is proved in~\cite{grigorch_z:basilica} that the group $\B$ is weakly regular branch over its commutator subgroup $\B'$, i.e.,
\begin{equation}
\label{eqn:basilica_branch}
\B'\succ \B'\times\B'.
\end{equation}
We show by induction on $l$ that $\B'$ acts transitively on $0X^{l-1}$ and $1X^{l-1}$. Induction base is trivial for $l=1$. Assume that the statement is true for a fixed value of $l$. Then equation~\eqref{eqn:basilica_branch} immediately implies that $\B'$ acts transitively on $00X^{l-1}$, $01X^{l-1}$, $10X^{l-1}$, and $11X^{l-1}$. Since $[a,b]=(a,b^{-1}a^{-1}b)\in\B'$ is such that both $a$ and $b^{-1}a^{-1}b$ act transitively on the first level, we obtain that $\B'$ acts transitively on $0X^l$ and $1X^l$.

To complete the proof we observe that since $\B'\times 1<\St_{\B}(1^l)$ for each $l$, the above argument implies that $\St_{\B}(1^l)$ acts transitively on $00X^{l-2}$ and $01X^{l-2}$. Finally, since $b=(a,1)\in\St_{\B}(1^l)$ is such that $b|_0=a$ acts transitively on the first level, we deduce that $\St_{\B}(1^l)$ acts transitively on $0X^{l-1}$ for all $l\geq 1$.
\end{proof}

\subsection{GGS-groups}
\label{ssec:ggs}
Important situation where one can prove liftability is when the projections of $\sigma(G)$ on the vertices of the first level, other than the distinguished vertex $i\in X$, happen to be in a finite subgroup of $G$. This is somewhat intermediate between the examples of groups acting essentially freely on the boundary of the tree (where the mentioned subgroup is trivial) and the examples of groups with finite $L$-presentation (where one needs to find the corresponding $L$-presentation). This situation is different from the case of groups acting essentially freely on the boundary of the tree that was discussed above. While the class of such groups is larger, it requires more work to check that the substitution $\sigma$ defined on the generators of $G$ extends to an endomorphism of $G$. In this section we demonstrate this situation on a large subclass of GGS-groups~\cite{bar_gs:branch,fernandez-alcober_z:GGS-groups_order_of_congruence_quotients14}. The class of GGS-groups (Grigorchuk-Gupta-Sidki groups, the term coined to G.~Baumslag) includes the group $\G$ (a second group from~\cite{grigorch:burnside}) and all Gupta-Sidki $p$-groups constructed in~\cite{gupta_s:burnside}. It was shown in~\cite{fernandez-alcober_z:GGS-groups_order_of_congruence_quotients14} that all GGS $p$-groups are regular branch, so they act totally non-freely on $\partial T_p$, which makes them very different from groups acting essentially freely on the boundaries of trees considered above.

For a fixed $d\geq 2$ every GGS-group is defined by a non-zero vector $\be=(e_1,e_2,\ldots,e_{d-1})\in(\Z/d\Z)^{d-1}$ as the group acting on the regular rooted tree $T_d=T(X)$ for $X=\Z/d\Z$ with the following wreath recursion:
\[
\begin{array}{rcl}
a&=&(1,1,\ldots,1,1)\varepsilon,\\
b&=&(a^{e_1},a^{e_2},\ldots,a^{e_{d-1}},b),
\end{array}
\]
where $\varepsilon=(0,1,\ldots,d-1)$ denotes a long cycle in $\Sym(X)$. We will denote the corresponding group as $G_{\be}=\langle a,b\rangle$.

\begin{theorem}
\label{thm:ggs}
Let $p\geq 3$ be a prime and $G_{\be}$ be a GGS-group defined by a non-symmetric vector $\be=(e_1,e_2,\ldots,e_{p-1})\in (\Z/p\Z)^{p-1}$ with the property that there is $j\in\{1,2,\ldots,p-1\}$ such that $e_j\neq 0$ and $e_{p-j}=0$. Then $G_{\be}$ is a liftable group.
\end{theorem}

\begin{proof}
We first observe that both generators $a$ and $b$ of $G_{\be}$ have order $p$. Moreover, it is shown in~\cite{fernandez-alcober_z:GGS-groups_order_of_congruence_quotients14} (Theorem~2.1) that the abelianization $G_{\be}/G_{\be}'$ is isomorphic to $(\Z/p\Z)^2$.

Let $f=e_{j}^{-1}$, so that $(a^{e_j})^f=a$. Then we have
\[
\begin{array}{rcllllllllll}
a^{-1}ba&=&(b,&a^{e_1},&\ldots,&a^{e_{p-j-1}},&1,&a^{e_{p-j+1}},&\ldots,&a^{e_{p-2}},&a^{e_{p-1}}&),\\
(a^{j-1}b{a^{1-j}})^f&=&(a,&a^{e_{j+1}f},&\ldots,&a^{e_{p-1}f},&b^{f},&a^{e_{1}f},&\ldots,&a^{e_{j-2}f},&a^{e_{j-1}f}&),\\[1mm]
\text{positions}&&\phantom{(}0&1&\ldots&p-j-1&p-j&p-j+1&\ldots&p-2&p-1&\\
\end{array}
\]
where the last row indicates the positions of corresponding coordinates of the vectors in the first two rows. Define a substitution $\sigma$ as
\begin{equation}
\label{eqn:sigma_GGS}
\sigma\colon\left\{
\begin{array}{l}
a\mapsto (a^{j-1}b{a^{1-j}})^f\\
b\mapsto a^{-1}ba\\
\end{array}\right.
\end{equation}
Then $\sigma$ extends to an injective endomorphism of $G_{\be}$. Indeed, suppose $w=w(a,b)$ is a word in the free group $F(a,b)$ that is a relator in $G_{\be}$. Then
\[\pi_i\circ\sigma(w)=\pi_i\bigl(w((a^{j-1}b{a^{1-j}})^f,a^{-1}ba)\bigr)=\left\{
\begin{array}{ll}
w(a,b)=_{G_{\be}}1,&\text{if }\ i=0,\\
w(1,b^f),&\text{if }\ i=p-j+1,\\
w(a^{e_i},a^{j+i}),&\text{otherwise.}
\end{array}
\right.\]
Since $G_{\be}/G_{\be}'\cong(\Z/p\Z)^2$ and $w=_{G_{\be}}1$, we must have that the total exponents of $a$ and $b$ in $w$ are both equal to 0 modulo $p$. This implies that both $w(1,b^f)$ and $w(a^{e_i},a^{j+i})$ are also trivial in $G_{\be}$ since both $a$ and $b$ have order $p$. Here we use the fact that $\pi_i\circ\sigma(G)$ is a subgroup of either $\langle b\rangle$ or $\langle a\rangle$ for $i\neq 0$, which are both finite subgroups of $G$, which makes the verification that $\sigma$ is an endomorphism possible.

Finally, by construction we have $\pi_0\circ\sigma$, so $\sigma$ is a lifting of $G_{\be}$ and thus $G_{\be}$ is liftable. In particular, Theorem~\ref{thmx:hnn} applies for $G_{\be}$.
\end{proof}

\begin{corollary}
Gupta-Sidki $p$-groups $GS_{p}$ are liftable for all prime $p\geq 5$.
\end{corollary}

\begin{proof}
The Gupta-Sidki $p$-group constructed in~\cite{gupta_s:burnside} is a GGS-groups defined by the following vector $(1,-1,0,\ldots,0)\in(\Z/p\Z)^{p-1}$, $p\geq 3$. This vector satisfies the condition from Theorem~\ref{thm:ggs} for $j=1$ for $p\geq 5$.
\end{proof}

It is not known whether Gupta-Sidki 3-group $G_3$ is liftable. This is related to the question~4.10 from~\cite{bartholdi:epimorphic03} about the existence of a finite ascending $L$-presentation for $GS_3$.

\subsection{The group $\GE$}

Now we consider the group $\GE$ studied by A.~Erschler in~\cite{erschler:boundary}, who proved that it has intermediate growth with the growth function growing faster than $\exp(n^\alpha)$ for each $0<\alpha<1$ and also provided quite sharp upper and lower bounds on the growth function. This group belongs to the family $\G_{\omega}, \omega\in\{0,1,2\}^\infty$ constructed by the first author in~\cite{grigorch:degrees}. Observe that the group $\G$ from Example~\ref{ex:grigorchuk} is isomorphic to $\G_{(012)^\infty}$. The group $\GE$ is determined by a periodic sequence $(01)^\infty$, acts on the binary rooted tree, and is generated by the following wreath recursion:
\begin{equation}
\label{eqn:GE}
\begin{array}{rcl}
a&=&(1,1)\varepsilon\\
b&=&(a,c)\\
c&=&(1,b)\\
d&=&(a,d)\\
\end{array},
\end{equation}

Similarly to $\G$ and the GGS-groups studied in Subsection~\ref{ssec:ggs}, the group $\GE$ is just-infinite branch and thus acts totally non-freely on $\partial T_2$.

\begin{proposition}
\label{thm:GE}
The group $\GE$ is liftable.
\end{proposition}

Observe that it is proved by Bartholdi and Nekrashevych in~\cite{bartholdi-n:quadratic1} that $\GE$ is an iterated monodromy group of a quadratic polynomial and is finitely $L$-presented. One can check that the presentation that comes from this result is based on a lifting and Theorem~\ref{thmx:hnn} can be applied. Instead we demonstrate another (practical) approach based on the use of the GAP package working when no presentation of a group is known.

We introduce some groups related to $\GE$ and use \verb"AutomGrp" GAP package~\cite{muntyan_s:automgrp} to perform our calculations. This is why we separated this example from others considered in other sections. All of our calculations can be verified by hands if needed. First, we observe that the following relations in $\GE$ hold:
\[a^2=b^2=c^2=d^2=bcd=(ac)^4=1.\]
Therefore, the group $C=\langle a,c\rangle$ is a quotient of the dihedral group $D_4$. On the other hand, the quotient of $C$ by its stabilizer of the third level has size 8:
\begin{verbatim}
gap> G:=AutomatonGroup("a=(1,1)(1,2),b=(a,c),c=(1,b),d=(a,d)");
< a, b, c, d >
gap> Size(PermGroupOnLevel(Group([a,c]),3));
8
\end{verbatim}
Therefore, $C\cong D_4$. Further, let $B=\langle b\rangle^{\GE}$ denote the normal closure of $b$ in $\GE$. Then, $\GE/B\cong C\cong D_4$ and $G=B\rtimes C$.

Let $H=\St_{\GE}(1)$ be the stabilizer of the first level in $\GE$. Then $[\GE:H]=2$ and there is a homomorphism
\begin{equation}
\begin{array}{llll}\psi\colon &H&\to&\GE\times\GE\\
&h&\mapsto&(\pi_0(h), \pi_1(h))
\end{array}
\end{equation}

The wreath recursion~\eqref{eqn:GE} implies that $aba=(c,a)$, so we define the substitution $\sigma$ on the generators of $\GE$ as
\begin{equation}
\label{eqn:sigma_GE}
\sigma\colon\left\{
\begin{array}{llll}
a&\mapsto& aba&=(c,a)\\
b&\mapsto& c&=(1,b)\\
c&\mapsto& b&=(a,c)\\
d&\mapsto& d&=(a,d)
\end{array}\right.
\end{equation}

We will show below that $\sigma$ extends to an injective endomorphism of $\GE$ (whose image is obviously a subgroup of $\St_{\GE}(1)$).

Define two subgroups $C_0=\langle (a,c),(c,a)\rangle$ and $\widetilde C_0=\langle (1,a),(1,c)\rangle$ of $\GE\times\GE$. Since $C=\langle a,c\rangle\cong D_4$ and the map $a\mapsto c, c\mapsto a$ induces an automorphism of $C$, we have that $C_0\cong\widetilde C_0\cong D_4$ as well.

\begin{lemma}
\label{lem:psiH}
The subgroup $\psi(H)<\GE\times\GE$ is the semidirect product $(B\times B)\rtimes C_0$. The subgroup $\widetilde C_0$ is complementary to $\psi(H)$, so that $\GE\times\GE=\psi(H)\rtimes\widetilde C_0$ and $[(\GE\times\GE):\psi(H)]=8$.
\end{lemma}
\begin{proof}
First, we observe that the stabilizer $H$ is generated by $\{b,c,d,b^a,c^a,d^a\}$, so
\[\psi(H)=\langle (1,b),(a,c),(a,d),(b,1),(c,a),(d,a)\rangle=\langle (1,b),(b,1),(a,c),(c,a)\rangle=(B\times B)\cdot C_0.\]
To confirm that $\psi(H)=(B\times B)\rtimes C_0$ we only need to check that $(B\times B)\cap C_0=\{1\}$, which follows from the fact that $\GE=B\rtimes C$ and that the projection of $C_0$ on both coordinates is equal to $C$.

Finally, since $\GE\times\GE=\langle \psi(H),(1,a),(1,c)\rangle=\psi(H)\cdot \widetilde C_0$, to prove that $\GE\times\GE=\psi(H)\rtimes\widetilde C_0$ we only need to check that $\psi(H)\cap\widetilde C_0=\{1\}$. This can be verified by considering the quotients of these groups by the stabilizers of the fourth level.

\begin{verbatim}
gap> G:=AutomatonGroup("a=(1,1)(1,2),b=(a,c),c=(1,b),d=(a,d)");
< a, b, c, d >
gap> AG_UseRewritingSystem(G);
gap> H:=StabilizerOfFirstLevel(G);
< b, c, d, a*b*a, a*c*a, a*d*a >
gap> A:=TreeAutomorphism([1,a],());
(1, a)
gap> C:=TreeAutomorphism([1,c],());
(1, c)
gap> Intersection(PermGroupOnLevel(H,4),PermGroupOnLevel(Group([A,C]),4));
Group(())
\end{verbatim}
\end{proof}

\begin{proof}[Proof of Proposition~\ref{thm:GE}]
Let $w$ be a word in the generators of $\GE$ such that $w=_{\GE}1$. Then
\[\psi(\sigma(w))=(U,W)=_{\GE\times\GE}(U,1),\]
where $U$ is a word over $\{a,c,a^{-1},c^{-1}\}$. Then
\[\psi(a\sigma(w)a)=(W,U)=_{\GE\times\GE}(1,U)\in\widetilde C_0.\]
But since by Lemma~\ref{lem:psiH} $\psi(H)\cap\widetilde C_0=\{1\}$, we get $U=_{\GE}1$ and so $\sigma(w)=_{\GE}1$. Thus, $\sigma$ extends to a monomorphism $\GE\to\GE$.
\end{proof}

\section{Scale groups coming from Bass-Serre theory}
\label{sec:bass}
One more way to construct an action of a group on $\widetilde T_{d+1}$ is via using Bass-Serre Theory. Recall, that if for a finitely generated group $G$ there is an injective endomorphism $\sigma\colon G\to H$ onto a subgroup $H$ of $G$ of finite index $d$, then the corresponding ascending HNN extension $\widetilde G=G*_{\sigma}$, whose presentation is given by equation~\eqref{eqn:hnn_pres}, acts on its Bass-Serre tree $T\cong\widetilde T_{d+1}$ by automorphisms preserving one of its ends. The set of vertices of $T$ is the set $\{xG\colon x\in\widetilde G\}$ of left cosets of $G$ in $\widetilde G$, where vertices $xG$ and $yG$ are connected by an oriented edge (going from $xG$ to $yG$) if and only if $yG=xt^{-1}G$. In this tree a vertex $xG$ has one outgoing edge connecting it to $xt^{-1}G$ and $d$ incoming edges starting at vertices $xg_stG$, $i=1,\ldots,d$, where $\{g_1,\ldots,g_d\}$ is a left transversal of $H$ in $G$ (for our purposes we do not need to talk about the inversions of edges). The group $\widetilde G$ acts on $T$ by multiplication on left, so this action is vertex transitive. It also preserves the end $\omega$ of the tree $T$ corresponding to the sequence of vertices of the form $t^{-n}G$, $n\geq 0$ as follows from the normal form for the elements of ascending HNN extensions. This situation is discussed by Minasyan and Osin in~\cite[Proposition~4.13]{minasyan_o:acylindrical_hyperbolicity15}.

The stabilizers of vertices in $T$ are conjugates of $G$ by elements of $\widetilde G$, with $G$ itself being the stabilizer of the vertex $\Lambda=G\in V(T)$. Therefore, the action of $\widetilde G$ on $T$ is faithful (so that $\widetilde G$ embeds into $\Aut(\widetilde T_{d+1})_{\omega}$) if and only if the core of $G$ in $\widetilde G$ (i.e., the maximal normal subgroup of $\widetilde G$ in $G$) is trivial:
\begin{equation}
\label{eqn:core}
\mathrm{Core}(G)=\bigcap_{x\in\widetilde G}xGx^{-1}=\{1\}.
\end{equation}
In practice it is easier to check an equivalent to~\eqref{eqn:core} condition: $\mathrm{Core}\left(\cap_{n=1}^\infty \sigma^n(G)\right)=\{1\}$ (it is equivalent since $\sigma^n(G)$ for each $n$ is a conjugate of $G$ by a power of $t$). By construction the stabilizer of the vertex $\Lambda$ corresponding to the coset $1\cdot G=G$ is equal to $G$. The above facts are summarized in the following theorem.

\begin{theorem}
\label{thm:bass_serre}
Let $G$ be a group with a subgroup $H$ isomorphic to $G$ and of finite index $d$ and let $\sigma\colon G\rightarrow H$ be an isomorphism. Let $\widetilde G=G*_{\sigma}$ be the ascending HNN-extension of $G$ with respect to $\sigma$, whose presentation is given in equation~\eqref{eqn:hnn_pres}. Then
\begin{itemize}
\item[(i)] $\widetilde G$ acts on its Bass-Serre tree $T$, which is isomorphic to $\widetilde T_{d+1}$, by automorphisms fixing one of its ends $\omega$.
\item[(ii)] The action of $\widetilde G$ on $T$ is vertex transitive with the stabilizer of the vertex $\Lambda$ of $T$ corresponding to the coset $1\cdot G$ equal to $G$.
\item[(iii)] The action of $\widetilde G$ on $T$ is faithful if and only if $\mathrm{Core}(G)=\{1\}$ (equivalently, if $\mathrm{Core}\left(\cap_{n=1}^\infty \sigma^n(G)\right)=\{1\}$). In this case $\widetilde G$ embeds into $\Aut(\widetilde T_{d+1})_{\omega}$ as a group acting vertex transitively on $\widetilde T_{d+1}$ and its closure in $\Aut(\widetilde T_{d+1})_{\omega}$ is a scale group. Moreover, the image of $G$ under this embedding coincides with the stabilizer of the distinguished vertex $\Lambda=1\cdot G$ of $T$.
\end{itemize}
\end{theorem}

\begin{remark}
The above theorem works also in the case $d=[G:\sigma(G)]=\infty$, but in this case we get an action of $\widetilde G=G*_{\sigma}$ on a tree $\widetilde T_{\infty}$ that is not locally finite.
\end{remark}

To use the above theorem we need to deal with groups that are not finitely co-Hopfian, i.e. that have finite index subgroups isomorphic to the whole group. The obvious examples of such groups are $\mathbb Z^d, d\geq 1$ (of course, we are interested in finitely generated examples). Less obvious examples come from scale-invariant groups that were defined at the beginning of Section~\ref{sec:examples}. To get embeddings of such groups into $\Aut(\widetilde T_{d+1})_{\omega}$ the last condition of Theorem~\ref{thm:bass_serre} (on faithfulness of the action) needs to be satisfied.

\begin{example}
\label{ex:lamplighter2}
Our first example illustrating the use of Theorem~\ref{thm:bass_serre} is the lamplighter group $\mathcal L=(\mathbb Z/2\mathbb Z)\wr\mathbb Z$, studied in Example~\ref{ex:lamplighter}. Recall that the standard presentation of $\mathcal L$ coming from the ``lamplighter'' interpretation of the action, is given with respect to a generating set $\{x=a^{-1}b, s=b\}$, where $\{a,b\}$ is the generating set from Example~\ref{ex:lamplighter}:
\[\mathcal L\cong\langle x,s\mid x^2=[x,x^{s^n}]=1, n\geq 1\rangle.\]
It was shown in~\cite{grigorch_lsz:atiyah} that the substitution $\alpha\colon x\mapsto xx^s, s\mapsto s$ extends to an injective endomorphism $\alpha\colon\mathcal L\to\mathcal L$, whose image has index 2 in $\mathcal L$. The homomorphism $\alpha$ written in terms of generating set $\{a,b\}$ coincides with the lifting $\sigma_1$ defined in~\eqref{eqn:sigma1_2_lamplighter}. Thus, the image of $\alpha$ is equal to $\St_{\mathcal L}(1)$ and has index $2$ in $\mathcal L$. It is also shown in~\cite{grigorch_lsz:atiyah} that HNN extension of $\mathcal L$ with respect to $\alpha$ is a metabelian group given by finite presentation
\[\widetilde{\mathcal L}_{\alpha}=\langle x,s,t\mid x^2=[t,s]=[x,x^{s}]=1,x^t=xx^s\rangle.\]
By Corollary~\ref{cor:ess_free} this group embeds into $\Aut(\widetilde T_3)_\omega$. However, Theorem~\ref{thm:bass_serre} can also be applied here by verifying that the core of $\mathcal L$ in $\widetilde{\mathcal L}$ is trivial. We start with the following lemma.

\begin{lemma}
\label{lem:alpha_image}
For each $n\geq 0$ we have
\[\alpha^{2^n}(x)=x\cdot x^{s^{2^n}},\]
\[\alpha^{2^n}(s)=s.\]
\end{lemma}

\begin{proof}
Since $\alpha(s)=s$, the second claim of the lemma is trivial. To prove the first claim we use induction on $n$. The base of induction is trivial and to prove the step we assume that the statement is correct for some $n\geq 0$. Then
\begin{multline*}
\alpha^{2^{n+1}}(x)=\alpha^{2^n}\left(\alpha^{2^n}(x)\right)=\alpha^{2^n}\left(x\cdot x^{s^{2^n}}\right)=x\cdot x^{s^{2^n}}\cdot \left(x\cdot x^{s^{2^n}}\right)^{s^{2^n}}\\
=x\cdot x^{s^{2^n}}\cdot x^{s^{2^n}}\cdot x^{s^{2^{n+1}}}=x\cdot x^{s^{2^{n+1}}}.
\end{multline*}
\end{proof}

\begin{lemma}
\label{lemma:core}
$\bigcap_{n=1}^\infty \alpha^n(\mathcal L)=\langle s\rangle$.
\end{lemma}
\begin{proof}
It is enough to show that $\bigcap_{n=1}^\infty \alpha^{2^n}(\mathcal L)=\langle s\rangle$. Since $s\in\alpha^{2^n}(\mathcal L)$ for all $n\geq 1$, one inclusion is obvious. We will show below that $\bigcap_{n=1}^\infty \alpha^{2^n}(\mathcal L)\subset\langle s\rangle$.

Every element $g$ of $\mathcal L$ can be uniquely written as
\begin{equation}
\label{eqn:prod_x}
g=\left(\prod_{i=1}^{k}x^{s^{n_i}}\right)\cdot s^l,
\end{equation}
where $k\in\Z$ and $n_1<n_2<\cdots<n_k$ are distinct integers. In the lamplighter interpretation such an element corresponds to turning on lamps at positions $n_i$ and moving the pointer to position $l$.

By Lemma~\ref{lem:alpha_image} $\alpha^{2^n}(\mathcal L)$ is generated by $\{x\cdot x^{s^{2^n}},s\}$. Define commuting elements
\[x_{2^n,i}=\left(x\cdot x^{s^{2^n}}\right)^{s^i}=x^{s^i}\cdot x^{s^{2^n}+i}, i\in\Z\]
Then every element $h$ of $\alpha^{2^n}(\mathcal L)$ can be written uniquely as
\[h=\left(\prod_{i=1}^{k}x_{2^n,m_i}\right)\cdot s^l,\]
where $l\in\Z$ and $m_1<m_2<\cdots<m_k$ are distinct integers. In the case $k\geq 1$ (i.e., when $h\notin\langle s\rangle$), the above decomposition rewritten in the form~\eqref{eqn:prod_x}, will satisfy $n_1=m_1$ and $n_k=m_k+2^n\geq n_1+2^n$. Therefore, every element in $\alpha^{2^n}(\mathcal L)\setminus\langle s\rangle$ will turn on 2 lamps that are at least $2^n$ units apart. This shows that there is no element in $\mathcal L\setminus\langle s\rangle$ that belongs to $\alpha^{2^n}(\mathcal L)$ for all $n\geq 1$.
\end{proof}

\begin{lemma}
$\mathrm{Core}\bigl(\bigcap_{n=1}^\infty \alpha^n(\mathcal L)\bigr)=\{1\}$.
\end{lemma}
\begin{proof}
According to Lemma~\ref{lemma:core} we only need to show that $\langle s\rangle$ does not contain a non-trivial normal subgroup in $\widetilde{\mathcal L}_{\alpha}$. This is the case because the conjugate $(s^i)^x$ of any non-trivial element $s^i$ of $\langle s\rangle$ is not a power of a shift $s$ as it changes the states of 2 lamps, so $(s^i)^x\notin \langle s\rangle$ and thus $\langle s\rangle$ contains no normal subgroup of $\widetilde{\mathcal L}_{\alpha}$.
\end{proof}

Now, since $\mathrm{Core}(\mathcal L)<\mathrm{Core}\bigl(\bigcap_{n=1}^\infty \alpha^n(\mathcal L)\bigr)$, the group $\widetilde{\mathcal L}_{\alpha}$ embeds into $\Aut(\widetilde T_{3})_{\omega}$ by Theorem~\ref{thm:bass_serre}.

The above example is interesting in particular because the group $\widetilde{\mathcal L}_{\alpha}$ is a quotient by the relation $x^2=1$ of Baumslag's finitely presented metabelian group with the commutator isomorphic to a free abelian group of infinite rank considered in~\cite{baumslag:fp_metabelian72}. Hence, $\widetilde{\mathcal L}_{\alpha}$ is a finitely presented metabelian group with the commutator subgroup isomorphic to an elementary 2-group of infinite rank.  Moreover, $\widetilde{\mathcal L}_{\alpha}$ is isomorphic to the fundamental group of a closed Riemannian manifold with a non-integer $L^2$-Betti number, constructed in~\cite{grigorch_lsz:atiyah} (as the first example answering the original question of M.~Atiyah and a counterexample to the Strong Atiyah conjecture on $L^2$-Betti numbers).
\end{example}

Formally, the two constructions of the actions of HNN extensions on $\widetilde T_{d+1}$ given by Theorem~\ref{thmx:hnn}
and Theorem~\ref{thm:bass_serre} are quite different and it is quite unlikely that they lead to conjugate action of the same group (although in some cases perhaps conjugacy of action is possible). Observe that in the construction used in the proof of Theorem~\ref{thm:bass_serre} we do not have a requirement that the index $[G:\sigma(G)]$ is related to the degree of the tree as we have in the case of Theorem~\ref{thm:bass_serre}. Moreover, this index in most examples presented in Section~\ref{sec:examples} is infinite. On the other hand, for each liftable self-similar group $G$ that acts transitively on the first level and is defined by a lifting $\sigma$, the core of $G$ in its HNN extension $\widetilde G$ from Theorem~\ref{thmx:hnn} is trivial. Indeed, since $G$ stabilizes the vertex $\Lambda$ in $\widetilde T_{d+1}$, its conjugates by elements of $\widetilde G$ are subgroups of the stabilizers of vertices of $\widetilde T_{d+1}$ in $\widetilde G$. Moreover, since $\widetilde G$ acts transitively on $\widetilde T_{d+1}$, one can conjugate $G$ into a stabilizer of any vertex. Now, as the action of $\widetilde G$ on $\widetilde T_{d+1}$ is faithful, the intersection of all these stabilizers, which includes $\mathrm{Core}(G)$ in $\widetilde G$, is trivial. Therefore, by Theorem~\ref{thm:bass_serre}, $\widetilde G$ acts on the corresponding Bass-Serre tree $\widetilde T_{d'+1}$, where $d'=[G:\sigma(G)]$ and in the case $d'<\infty$ this leads to the constructions of scale groups via Theorem~\ref{thm:bass_serre} (like in the last example).



\section{Scale-invariant groups and Haagerup property}
\label{sec:haagerup}

We are going to explore scale-invariant groups to get more examples of scale groups using Theorem~\ref{thm:bass_serre}. The main result of Nekrashevych and Pete in~\cite{nekrashevych_p:scale_invariant} and its corollary state:

\begin{theorem}[Theorem~1.1 in~\cite{nekrashevych_p:scale_invariant}]
\label{thm:scale}
Let $H$ be a countable scale-invariant group, with a descending sequence
$H=H_0 > H_1 >\cdots$  of finite index subgroups, each isomorphic to $H$,
with trivial intersection. Let $A$ be a countable automorphism group of $H$ leaving all subgroups $H_n$ invariant. Assume that the action of $A$ is faithful on each $H_n$, and that the semidirect products $H_n\rtimes A$ are isomorphic to $H\rtimes A$. Then $G:=H\rtimes A$ is scale-invariant; in fact, there is a required subgroup chain with $\cap_{n\geq 0}G_n=\{1\}$.
\end{theorem}

\begin{corollary}[Corollary~1.2 in~\cite{nekrashevych_p:scale_invariant}]  The following groups are scale-invariant.
\begin{itemize}
\item[(1)] The lamplighter groups $G=F\wr\Z$, where $F$ is any finite Abelian group.
\item[(2)] The solvable Baumslag–Solitar groups $BS(1,m)=\langle a, b\mid bab^{-1} = a^m\rangle$ with $m > 1$.
\item[(3)] The affine group $\Z^d\rtimes\mathrm{GL}(\Z, d)$, and its subgroups $\Z^d\rtimes A$ for any $A\leq \mathrm{GL}(\Z, d), d > 1$.
\end{itemize}
\end{corollary}

The proof of the theorem uses (rooted) coset tree $\mathcal T$ associated with $\{H_n\}_{n\geq 1}$ and the natural level transitive action of $H\rtimes A$ on it by affine automorphisms. The sequence $\{G_n\}_{n\geq 1}$ is obtained as the sequence of stabilizers of vertices belonging to the path in $\mathcal T$ joining the root vertex with infinity (i.e., with an end of the tree). It is proved that the stabilizers are isomorphic to $G$. The fact that they have finite index in $G$ is obvious from the construction, and the main nontrivial point is to show that there are points in the boundary of $\mathcal T$ with trivial stabilizers. The arguments provided in~\cite{nekrashevych_p:scale_invariant} show that in fact the set of such points is a dense $G_\delta$ set, so is non-empty (and comeager).

For our purposes we will need Proposition~2.2(i) from~\cite{nekrashevych_p:scale_invariant} and its proof. It states that if, under conditions of Theorem~\ref{thm:scale}, the sequence $\{H_n\}_{n\geq1}$ satisfies $H_n=\varphi^n(H)$ for some injective endomorphism $\varphi\colon H\to H$, then there is an injective endomorphism $\tilde\varphi\colon G\to G$ defined by (see equation~(2.3) in~\cite{nekrashevych_p:scale_invariant})
\begin{equation}
\label{eqn:varphi}
\tilde\varphi(h,\alpha)=\bigl(\varphi(h),\alpha\bigr)
\end{equation}
for $(h,\alpha)\in G$ such that $[G:\tilde\varphi(G)]=[H:\varphi(H)]<\infty$, and such that $\tilde\varphi^n(G)$ is the stabilizer in $G$ of a vertex of $\mathcal T$ of level $n$. Let
\[\widetilde G=G*_{\tilde\varphi}=\langle G,t\mid \text{relations in}\ G, tgt^{-1}=\tilde\varphi(g)\ \text{for all}\ g\in G\rangle\]
be the ascending HNN extension of $G$ with respect to $\tilde\varphi$.

\begin{proposition}
The core of $G$ in $\widetilde G$ is trivial.
\end{proposition}

\begin{proof}
By the proof of Proposition~2.2(i) in~\cite{nekrashevych_p:scale_invariant}, $t^nGt^{-n}=\tilde\varphi^n(G)$ is the stabilizer in $G$ of a vertex of $\mathcal T$ of level $n$. Since the action of $G$ on $\mathcal T$ is level transitive, one can use conjugate $\tilde\varphi^n(G)$ by elements of $G$ into a stabilizer of any vertex of that (arbitrary) level. Since the action of $G$ on $\mathcal T$ is faithful, the intersection of all vertex stabilizers is trivial, which implies that $\mathrm{Core}(G)$, that is the intersection of all conjugates of $G$ in $\widetilde G$ is also trivial.
\end{proof}

As a direct corollary of the above proposition and Theorem~\ref{thm:bass_serre} we obtain:

\begin{corollary}
\label{cor:scale-invar}
The group $\widetilde G$ embeds into $\Aut(\widetilde T_{d+1})_{\omega}$ for $d=[G:\tilde\varphi(G)]$.
\end{corollary}

We apply Corollary~\ref{cor:scale-invar} to get examples of scale groups obtained as a closures of groups without Haagerup property. Recall that a (discrete) group $G$ has a \emph{Haagerup property} (or is \emph{a-T-menable}) if it admits a proper affine action on a Hilbert space. There are several equivalent definitions of this property (see, for instance, Section~1.1.1 and Theorem~2.1.1 in~\cite{cherix_cjv:groups_with_haagerup_property01}). Also, there are many definitions of \emph{Kazhdan's property (T)} and its relative version for a pair consisting of a group and its subgroup (see, for example, \cite[Definition~1.4.1]{cherix_cjv:groups_with_haagerup_property01}). Property (T) groups also have Serre's property (FA) meaning that every action on a tree has a fixed vertex. The class of Haagerup groups includes amenable groups, free groups, groups acting properly on trees, $\R$-trees, and $\mathrm{CAT}(0)$ cube complexes. Examples of groups with property (T) mostly come from lattices in higher rank Lie groups. A nontrivial example of a pair with relative property (T) is $(\Z^2\rtimes\mathrm{SL}(2,\Z),\Z^2)$ (see, for example,~\cite[Theorem~4.2.2]{bekka_hv:kazhdan_preperty_t08}), which, in particular, implies that $\Z^2\rtimes\mathrm{SL}(2,\Z)$ does not have Haagerup property.

\begin{theorem}
\begin{itemize}
\item[(a)]
There is a subgroup of $\Aut(\widetilde T_9)_\omega$ without the Haagerup property that acts vertex transitively on $\widetilde T_9$ and whose closure is a scale group.
\item[(b)]
There is a subgroup of $\Aut(\widetilde T_5)_\omega$ without the Haagerup property that acts vertex transitively on $\widetilde T_5$ and whose closure is a scale group.
\end{itemize}
\end{theorem}

\begin{proof}
To prove item (a), consider a scale-invariant group $H=\Z^3$ with an injective endomorphism $\varphi\colon H\to H$ defined by $\varphi(v)=2v$. This defines a nested sequence $H_n=\varphi^n(H)$ of subgroups of $H$ with trivial intersection such that $[H_n:H_{n+1}]=8$ for all $n\geq0$. Then according to the proof of Theorem~\ref{thm:scale} the group $\Gamma=\Z^3\rtimes\mathrm{SL}(\Z,3)$ consisting of affine transformations of $\Z^3$, which is a group without Haagerup property (as the class of groups with Haagerup property is closed under taking subgroups), acts faithfully by automorphisms on the coset tree $\mathcal T$, that is a regular rooted 8-ary tree. By Corollary~\ref{cor:scale-invar} its HNN extension $\widetilde\Gamma*_{\tilde\varphi}$, where $\tilde\varphi$ is defined by equation~\eqref{eqn:varphi}, embeds into $\Aut(\widetilde T_9)_\omega$. Its closure in $\Aut(\widetilde T_9)_\omega$ is a scale group by Theorem~\ref{thm:bass_serre}(iii).

The proof of item (b) is very similar with the difference that we take $H=\Z^2$ and $\Gamma=\Z^2\rtimes\mathrm{SL}(\Z,2)$ as the corresponding scale-invariant group. Since now $[H:\varphi(H)]=4$, the HNN extension $\widetilde\Gamma*_{\tilde\varphi}$, where $\tilde\varphi$ is defined by equation~\eqref{eqn:varphi}, embeds into $\Aut(\widetilde T_5)_\omega$. Its closure in $\Aut(\widetilde T_5)_\omega$ is again a scale group by Theorem~\ref{thm:bass_serre}(iii).

\end{proof}




\begin{thebibliography}{10}

\bibitem{amir_av:amenab_linear_autom}
Gideon Amir, Omer Angel, and B{\'a}lint Vir{\'a}g.
\newblock Amenability of linear-activity automaton groups.
\newblock {\em J. Eur. Math. Soc. (JEMS)}, 15(3):705--730, 2013.

\bibitem{bartholdi_g:spectrum}
L.~Bartholdi and R.~I. Grigorchuk.
\newblock On the spectrum of {H}ecke type operators related to some fractal
  groups.
\newblock {\em Tr. Mat. Inst. Steklova}, 231(Din. Sist., Avtom. i Beskon.
  Gruppy):5--45, 2000.

\bibitem{bartholdi:epimorphic03}
Laurent Bartholdi.
\newblock Endomorphic presentations of branch groups.
\newblock {\em J. Algebra}, 268(2):419--443, 2003.

\bibitem{bartholdi_gn:fractal}
Laurent Bartholdi, Rostislav Grigorchuk, and Volodymyr Nekrashevych.
\newblock From fractal groups to fractal sets.
\newblock In {\em Fractals in Graz 2001}, Trends Math., pages 25--118.
  Birkh\"auser, Basel, 2003.

\bibitem{bartholdi_g:parabolic}
Laurent Bartholdi and Rostislav~I. Grigorchuk.
\newblock On parabolic subgroups and {H}ecke algebras of some fractal groups.
\newblock {\em Serdica Math. J.}, 28(1):47--90, 2002.

\bibitem{bar_gs:branch}
Laurent Bartholdi, Rostislav~I. Grigorchuk, and Zoran {\v{S}}uni{\'k}.
\newblock Branch groups.
\newblock In {\em Handbook of algebra, Vol. 3}, pages 989--1112. North-Holland,
  Amsterdam, 2003.

\bibitem{bkn:amenab}
Laurent Bartholdi, Vadim~A. Kaimanovich, and Volodymyr~V. Nekrashevych.
\newblock On amenability of automata groups.
\newblock {\em Duke Math. J.}, 154(3):575--598, 2010.

\bibitem{bartholdi-n:quadratic1}
Laurent Bartholdi and Volodymyr~V. Nekrashevych.
\newblock Iterated monodromy groups of quadratic polynomials. {I}.
\newblock {\em Groups Geom. Dyn.}, 2(3):309--336, 2008.

\bibitem{bartholdi_nw:horocyclic_product_of_trees08}
Laurent Bartholdi, Markus Neuhauser, and Wolfgang Woess.
\newblock Horocyclic products of trees.
\newblock {\em J. Eur. Math. Soc. (JEMS)}, 10(3):771--816, 2008.

\bibitem{bartholdi_v:amenab}
Laurent Bartholdi and B{\'a}lint Vir{\'a}g.
\newblock Amenability via random walks.
\newblock {\em Duke Math. J.}, 130(1):39--56, 2005.
\newblock (available at \emph{http://arxiv.org/abs/math.GR/0305262}).

\bibitem{bartholdi_s:bsolitar}
Laurent~I. Bartholdi and Zoran {\v{S}}uni{\'k}.
\newblock Some solvable automaton groups.
\newblock In {\em Topological and Asymptotic Aspects of Group Theory}, volume
  394 of {\em Contemp. Math.}, pages 11--29. Amer. Math. Soc., Providence, RI,
  2006.

\bibitem{bass:covering_theory93}
Hyman Bass.
\newblock Covering theory for graphs of groups.
\newblock {\em J. Pure Appl. Algebra}, 89(1-2):3--47, 1993.

\bibitem{bass_l:rigidity94}
Hyman Bass and Alexander Lubotzky.
\newblock Rigidity of group actions on locally finite trees.
\newblock {\em Proc. London Math. Soc. (3)}, 69(3):541--575, 1994.

\bibitem{baumslag:fp_metabelian72}
Gilbert Baumslag.
\newblock A finitely presented metabelian group with a free abelian derived
  group of infinite rank.
\newblock {\em Proc. Amer. Math. Soc.}, 35:61--62, 1972.

\bibitem{baumslag:topics_book93}
Gilbert Baumslag.
\newblock {\em Topics in combinatorial group theory}.
\newblock Lectures in Mathematics ETH Z\"{u}rich. Birkh\"{a}user Verlag, Basel,
  1993.

\bibitem{bekka_hv:kazhdan_preperty_t08}
Bachir Bekka, Pierre de~la Harpe, and Alain Valette.
\newblock {\em Kazhdan's property ({T})}, volume~11 of {\em New Mathematical
  Monographs}.
\newblock Cambridge University Press, Cambridge, 2008.

\bibitem{bekka_h:irreducibility_of_unitary03}
M.~Bachir Bekka and Pierre de~la Harpe.
\newblock Irreducibility of unitary group representations and reproducing
  kernels {H}ilbert spaces.
\newblock {\em Expo. Math.}, 21(2):115--149, 2003.
\newblock Appendix by the authors in collaboration with Rostislav Grigorchuk.

\bibitem{bier_ls:automorphisms_of_parabolic_trees16}
Agnieszka Bier, Yuriy Leshchenko, and Vitaliy Sushchanskyy.
\newblock Automorphisms of restricted parabolic trees and {S}ylow
  {$p$}-subgroups of the finitary symmetric group.
\newblock {\em J. Algebra}, 452:401--426, 2016.

\bibitem{bodart:rat_cross-sections}
Corentin Bodart.
\newblock Rational cross-sections, bounded generation and orders on groups.
\newblock arxiv:2210.04219, 2022.

\bibitem{bondarenko_gkmnss:full_clas32_short}
I.~Bondarenko, R.~Grigorchuk, R.~Kravchenko, Y.~Muntyan, V.~Nekrashevych,
  D.~Savchuk, and Z.~\v{S}uni\'{c}.
\newblock Classification of groups generated by $3$-state automata over
  $2$-letter alphabet.
\newblock {\em Algebra Discrete Math.}, (1):1--163, 2008.
\newblock (available at \emph{http://arxiv.org/abs/0803.3555}).

\bibitem{burger_m:groups_acting_on_trees00}
Marc Burger and Shahar Mozes.
\newblock Groups acting on trees: from local to global structure.
\newblock {\em Inst. Hautes \'{E}tudes Sci. Publ. Math.}, (92):113--150 (2001),
  2000.

\bibitem{button:groups_acting_faithfully_on_trees}
Jack~Oliver Button.
\newblock Groups acting faithfully on trees and properly on products of trees.
\newblock arxiv:1910.04614, 2019.

\bibitem{bux_p:iter_monodromy}
Kai-Uwe Bux and Rodrigo P{\'e}rez.
\newblock On the growth of iterated monodromy groups.
\newblock In {\em Topological and asymptotic aspects of group theory}, volume
  394 of {\em Contemp. Math.}, pages 61--76. Amer. Math. Soc., Providence, RI,
  2006.
\newblock (available at \emph{http://www.arxiv.org/abs/math.GR/0405456}).

\bibitem{caprace_w:tdlc_invitation}
Pierre-Emmanuel Caprace and George~A. Willis.
\newblock A totally disconnected invitation to locally compact groups.
\newblock Preprint: arxiv:2110.05991, 2021.

\bibitem{cartwright_kw:random_walks_on_the_affine_group94}
D.~I. Cartwright, V.~A. Ka\u{\i}manovich, and W.~Woess.
\newblock Random walks on the affine group of local fields and of homogeneous
  trees.
\newblock {\em Ann. Inst. Fourier (Grenoble)}, 44(4):1243--1288, 1994.

\bibitem{cherix_cjv:groups_with_haagerup_property01}
Pierre-Alain Cherix, Michael Cowling, Paul Jolissaint, Pierre Julg, and Alain
  Valette.
\newblock {\em Groups with the {H}aagerup property}, volume 197 of {\em
  Progress in Mathematics}.
\newblock Birkh\"{a}user Verlag, Basel, 2001.
\newblock Gromov's a-T-menability.

\bibitem{cull_n:hanoi_codes99}
Paul Cull and Ingrid Nelson.
\newblock Error-correcting codes on the {T}owers of {H}anoi graphs.
\newblock {\em Discrete Math.}, 208/209:157--175, 1999.
\newblock Combinatorics (Assisi, 1996).

\bibitem{culler_m:group_actions_on_R_trees87}
Marc Culler and John~W. Morgan.
\newblock Group actions on {${\bf R}$}-trees.
\newblock {\em Proc. London Math. Soc. (3)}, 55(3):571--604, 1987.

\bibitem{harpe_cg:paradoxical}
P.~de~lya Arp, R.~I. Grigorchuk, and T.~Chekerini-Sil{\cprime}bersta{\u\i}n.
\newblock Amenability and paradoxical decompositions for pseudogroups and
  discrete metric spaces.
\newblock {\em Tr. Mat. Inst. Steklova}, 224(Algebra. Topol. Differ. Uravn. i
  ikh Prilozh.):68--111, 1999.

\bibitem{dicks_d:groups_acting_on_graphs89}
Warren Dicks and M.~J. Dunwoody.
\newblock {\em Groups acting on graphs}, volume~17 of {\em Cambridge Studies in
  Advanced Mathematics}.
\newblock Cambridge University Press, Cambridge, 1989.

\bibitem{erschler:boundary}
Anna Erschler.
\newblock Boundary behavior for groups of subexponential growth.
\newblock {\em Ann. of Math. (2)}, 160(3):1183--1210, 2004.

\bibitem{fernandez-alcober_z:GGS-groups_order_of_congruence_quotients14}
Gustavo~A. Fern\'{a}ndez-Alcober and Amaia Zugadi-Reizabal.
\newblock G{GS}-groups: order of congruence quotients and {H}ausdorff
  dimension.
\newblock {\em Trans. Amer. Math. Soc.}, 366(4):1993--2017, 2014.

\bibitem{garzon_z:crypto}
Max Garzon and Yechezkel Zalcstein.
\newblock The complexity of {G}rigorchuk groups with application to
  cryptography.
\newblock {\em Theoret. Comput. Sci.}, 88(1):83--98, 1991.

\bibitem{grigorch_k:lamplighter}
R.~Grigorchuk and R.~Kravchenko.
\newblock On the lattice of subgroups of the lamplighter group.
\newblock {\em Internat. J. Algebra Comput.}, 24(6):837--877, 2014.

\bibitem{grigorch:burnside}
R.~I. Grigorchuk.
\newblock On {B}urnside's problem on periodic groups.
\newblock {\em Funktsional. Anal. i Prilozhen.}, 14(1):53--54, 1980.

\bibitem{grigorch:degrees}
R.~I. Grigorchuk.
\newblock Degrees of growth of finitely generated groups and the theory of
  invariant means.
\newblock {\em Izv. Akad. Nauk SSSR Ser. Mat.}, 48(5):939--985, 1984.

\bibitem{grigorch:hilbert}
R.~I. Grigorchuk.
\newblock On the {H}ilbert-{P}oincar\'e series of graded algebras that are
  associated with groups.
\newblock {\em Mat. Sb.}, 180(2):207--225, 304, 1989.

\bibitem{grigorch:example}
R.~I. Grigorchuk.
\newblock An example of a finitely presented amenable group that does not
  belong to the class {EG}.
\newblock {\em Mat. Sb.}, 189(1):79--100, 1998.

\bibitem{gr99:schur}
R.~I. Grigorchuk.
\newblock On the system of defining relations and the {S}chur multiplier of
  periodic groups generated by finite automata.
\newblock In {\em Groups St. Andrews 1997 in Bath, I}, volume 260 of {\em
  London Math. Soc. Lecture Note Ser.}, pages 290--317. Cambridge Univ. Press,
  Cambridge, 1999.

\bibitem{grigorch:branch}
R.~I. Grigorchuk.
\newblock Branch groups.
\newblock {\em Mat. Zametki}, 67(6):852--858, 2000.

\bibitem{grigorch:jibranch}
R.~I. Grigorchuk.
\newblock Just infinite branch groups.
\newblock In {\em New horizons in pro-$p$ groups}, volume 184 of {\em Progr.
  Math.}, pages 121--179. Birkh\"auser Boston, Boston, MA, 2000.

\bibitem{grigorch:dynamics11eng}
R.~I Grigorchuk.
\newblock Some topics in the dynamics of group actions on rooted trees.
\newblock {\em Proc. of Steklov Inst. of Math.}, 273:64--175, 2011.

\bibitem{grigorch_hz:profinite_completions00}
R.~I. Grigorchuk, W.~N. Herfort, and P.~A. Zalesskii.
\newblock The profinite completion of certain torsion {$p$}-groups.
\newblock In {\em Algebra ({M}oscow, 1998)}, pages 113--123. de Gruyter,
  Berlin, 2000.

\bibitem{gns00:automata}
R.~I. Grigorchuk, V.~V. Nekrashevich, and V.~I. Sushchanski{\u\i}.
\newblock Automata, dynamical systems, and groups.
\newblock {\em Tr. Mat. Inst. Steklova}, 231(Din. Sist., Avtom. i Beskon.
  Gruppy):134--214, 2000.

\bibitem{grigorch:solved}
Rostislav Grigorchuk.
\newblock Solved and unsolved problems around one group.
\newblock In {\em Infinite groups: geometric, combinatorial and dynamical
  aspects}, volume 248 of {\em Progr. Math.}, pages 117--218. Birkh\"auser,
  Basel, 2005.

\bibitem{grigorch_g:key_agreement19}
Rostislav Grigorchuk and Dima Grigoriev.
\newblock Key agreement based on automaton groups.
\newblock {\em Groups Complex. Cryptol.}, 11(2):77--81, 2019.

\bibitem{grigorch_ln:spectra_schreier_schroedinger18}
Rostislav Grigorchuk, Daniel Lenz, and Tatiana Nagnibeda.
\newblock Spectra of {S}chreier graphs of {G}rigorchuk's group and
  {S}chroedinger operators with aperiodic order.
\newblock {\em Math. Ann.}, 370(3-4):1607--1637, 2018.

\bibitem{grigorch_s:essfree}
Rostislav Grigorchuk and Dmytro Savchuk.
\newblock Self-similar groups acting essentially freely on the boundary of the
  binary rooted tree.
\newblock In {\em Group theory, combinatorics, and computing}, volume 611 of
  {\em Contemp. Math.}, pages 9--48. Amer. Math. Soc., Providence, RI, 2014.

\bibitem{grigorch_s:ergodic_decomposition}
Rostislav Grigorchuk and Dmytro Savchuk.
\newblock Ergodic decomposition of group actions on rooted trees.
\newblock {\em Tr. Mat. Inst. Steklova}, 292(Algebra, Geometriya i Teoriya
  Chisel):100--117, 2016.

\bibitem{grigorch_ss:img}
Rostislav Grigorchuk, Dmytro Savchuk, and Zoran {\v{S}}uni{\'c}.
\newblock The spectral problem, substitutions and iterated monodromy.
\newblock In {\em Probability and mathematical physics}, volume~42 of {\em CRM
  Proc. Lecture Notes}, pages 225--248. Amer. Math. Soc., Providence, RI, 2007.

\bibitem{grigorch_s:hanoi_spectrum}
Rostislav Grigorchuk and Zoran {\v{S}}uni{\'c}.
\newblock Schreier spectrum of the {H}anoi {T}owers group on three pegs.
\newblock In {\em Analysis on graphs and its applications}, volume~77 of {\em
  Proc. Sympos. Pure Math.}, pages 183--198. Amer. Math. Soc., Providence, RI,
  2008.

\bibitem{grigorchuk-s:hanoi-cr}
Rostislav Grigorchuk and Zoran {\v{S}}uni{\'k}.
\newblock Asymptotic aspects of {S}chreier graphs and {H}anoi {T}owers groups.
\newblock {\em C. R. Math. Acad. Sci. Paris}, 342(8):545--550, 2006.

\bibitem{grigorch_lsz:atiyah}
Rostislav~I. Grigorchuk, Peter Linnell, Thomas Schick, and Andrzej {\.Z}uk.
\newblock On a question of {A}tiyah.
\newblock {\em C. R. Acad. Sci. Paris S\'er. I Math.}, 331(9):663--668, 2000.

\bibitem{grigorch_z:lamplighter}
Rostislav~I. Grigorchuk and Andrzej {\.Z}uk.
\newblock The lamplighter group as a group generated by a 2-state automaton,
  and its spectrum.
\newblock {\em Geom. Dedicata}, 87(1-3):209--244, 2001.

\bibitem{grigorch_z:basilica}
Rostislav~I. Grigorchuk and Andrzej {\.Z}uk.
\newblock On a torsion-free weakly branch group defined by a three state
  automaton.
\newblock {\em Internat. J. Algebra Comput.}, 12(1-2):223--246, 2002.

\bibitem{grigorch_z:basilica_sp}
Rostislav~I. Grigorchuk and Andrzej {\.Z}uk.
\newblock Spectral properties of a torsion-free weakly branch group defined by
  a three state automaton.
\newblock In {\em Computational and statistical group theory (Las Vegas,
  NV/Hoboken, NJ, 2001)}, volume 298 of {\em Contemp. Math.}, pages 57--82.
  Amer. Math. Soc., Providence, RI, 2002.

\bibitem{gupta_s:burnside}
Narain Gupta and Said Sidki.
\newblock On the {B}urnside problem for periodic groups.
\newblock {\em Math. Z.}, 182(3):385--388, 1983.

\bibitem{kahrobaei_ms:lba}
Delaram Kahrobaei, Arsalan Malik, and Dmytro Savchuk.
\newblock Contracting self-similar groups in group-based cryptography.
\newblock Submitted, arXiv:2408.14355, 2024.

\bibitem{kaimanovich:self-similarity_and_random_walks09}
Vadim~A. Kaimanovich.
\newblock Self-similarity and random walks.
\newblock In {\em Fractal geometry and stochastics {IV}}, volume~61 of {\em
  Progr. Probab.}, pages 45--70. Birkh\"{a}user Verlag, Basel, 2009.

\bibitem{kambites-s-s:spectra}
Mark Kambites, Pedro~V. Silva, and Benjamin Steinberg.
\newblock The spectra of lamplighter groups and {C}ayley machines.
\newblock {\em Geom. Dedicata}, 120:193--227, 2006.

\bibitem{lysionok:presentation}
I.~G. Lys{\"e}nok.
\newblock A set of defining relations for the {G}rigorchuk group.
\newblock {\em Mat. Zametki}, 38(4):503--516, 634, 1985.

\bibitem{miasnikov_s:automatic_graph}
Alexei Miasnikov and Dmytro Savchuk.
\newblock An example of an automatic graph of intermediate growth.
\newblock {\em Ann. Pure Appl. Logic}, 166(10):1037--1048, 2015.

\bibitem{miasnikov_s:cayley_automatic11}
Alexei Miasnikov and Zoran {\v{S}}uni{\'c}.
\newblock Cayley graph automatic groups are not necessarily {C}ayley graph
  biautomatic.
\newblock In {\em Language and automata theory and applications}, volume 7183
  of {\em Lecture Notes in Comput. Sci.}, pages 401--407. Springer, Heidelberg,
  2012.

\bibitem{milne:ANT}
James~S. Milne.
\newblock Algebraic number theory (v3.08), 2020.
\newblock Available at \emph{http://www.jmilne.org/math/}.

\bibitem{minasyan_o:acylindrical_hyperbolicity15}
Ashot Minasyan and Denis Osin.
\newblock Acylindrical hyperbolicity of groups acting on trees.
\newblock {\em Math. Ann.}, 362(3-4):1055--1105, 2015.

\bibitem{muntyan_s:automgrp}
Y.~Muntyan and D.~Savchuk.
\newblock {\em {AutomGrp -- \verb+GAP+ package for computations in self-similar
  groups and semigroups, Version 1.3.2}}, 2019.
\newblock {A}ccepted \verb+GAP+ package (available at
  \emph{http://www.gap-system.org/Packages/automgrp.html}).

\bibitem{myasnikov_su:non_commutative_crypto_book11}
Alexei Myasnikov, Vladimir Shpilrain, and Alexander Ushakov.
\newblock {\em Non-commutative cryptography and complexity of group-theoretic
  problems}, volume 177 of {\em Mathematical Surveys and Monographs}.
\newblock American Mathematical Society, Providence, RI, 2011.
\newblock With an appendix by Natalia Mosina.

\bibitem{myasnikov_u:random_subgroups08}
Alexei~G. Myasnikov and Alexander Ushakov.
\newblock Random subgroups and analysis of the length-based and quotient
  attacks.
\newblock {\em J. Math. Cryptol.}, 2(1):29--61, 2008.

\bibitem{nebbia:amenability_and_kunze-stein_property88}
Claudio Nebbia.
\newblock Amenability and {K}unze-{S}tein property for groups acting on a tree.
\newblock {\em Pacific J. Math.}, 135(2):371--380, 1988.

\bibitem{nekrash_s:12endomorph}
V.~Nekrashevych and S.~Sidki.
\newblock Automorphisms of the binary tree: state-closed subgroups and dynamics
  of $1/2$-endomorphisms.
\newblock volume 311 of {\em London Math. Soc. Lect. Note Ser.}, pages
  375--404. {Cambridge Univ. Press}, 2004.

\bibitem{nekrash:self-similar}
Volodymyr Nekrashevych.
\newblock {\em Self-similar groups}, volume 117 of {\em Mathematical Surveys
  and Monographs}.
\newblock American Mathematical Society, Providence, RI, 2005.

\bibitem{nekrashevych_p:scale_invariant}
Volodymyr Nekrashevych and G{\'a}bor Pete.
\newblock Scale-invariant groups.
\newblock {\em Groups Geom. Dyn.}, 5(1):139--167, 2011.

\bibitem{otobe:locally_compact_fields45}
Yosikazu Otobe.
\newblock On locally compact fields.
\newblock {\em Jpn. J. Math.}, 19:189--202, 1945.

\bibitem{pays_v:sous-group_libres91}
Isabelle Pays and Alain Valette.
\newblock Sous-groupes libres dans les groupes d'automorphismes d'arbres.
\newblock {\em Enseign. Math. (2)}, 37(1-2):151--174, 1991.

\bibitem{petrides:cryptoanalysis_grigorchuk}
George Petrides.
\newblock Cryptanalysis of the public key cryptosystem based on the word
  problem on the {G}rigorchuk groups.
\newblock In {\em Cryptography and coding}, volume 2898 of {\em Lecture Notes
  in Comput. Sci.}, pages 234--244. Springer, Berlin, 2003.

\bibitem{pontryagin:topological_groups_3rd_edition_translated86}
L.~S. Pontryagin.
\newblock {\em Selected works. {V}ol. 2}.
\newblock Classics of Soviet Mathematics. Gordon \& Breach Science Publishers,
  New York, third edition, 1986.
\newblock Topological groups, Edited and with a preface by R. V. Gamkrelidze,
  Translated from the Russian and with a preface by Arlen Brown, With
  additional material translated by P. S. V. Naidu.

\bibitem{savchuk_v:free_prods}
Dmytro Savchuk and Yaroslav Vorobets.
\newblock Automata generating free products of groups of order 2.
\newblock {\em J. Algebra}, 336(1):53--66, 2011.

\bibitem{serre:arbres77}
Jean-Pierre Serre.
\newblock {\em Arbres, amalgames, {${\rm SL}_{2}$}}.
\newblock Ast\'{e}risque, No. 46. Soci\'{e}t\'{e} Math\'{e}matique de France,
  Paris, 1977.
\newblock Avec un sommaire anglais, R\'{e}dig\'{e} avec la collaboration de
  Hyman Bass.

\bibitem{serre:local_fields79}
Jean-Pierre Serre.
\newblock {\em Local fields}, volume~67 of {\em Graduate Texts in Mathematics}.
\newblock Springer-Verlag, New York-Berlin, 1979.
\newblock Translated from the French by Marvin Jay Greenberg.

\bibitem{serre:trees}
Jean-Pierre Serre.
\newblock {\em Trees}.
\newblock Springer Monographs in Mathematics. Springer-Verlag, Berlin, 2003.

\bibitem{vorobets:aleshin}
Mariya Vorobets and Yaroslav Vorobets.
\newblock On a free group of transformations defined by an automaton.
\newblock {\em Geom. Dedicata}, 124:237--249, 2007.

\bibitem{vorobets:series_free}
Mariya Vorobets and Yaroslav Vorobets.
\newblock On a series of finite automata defining free transformation groups.
\newblock {\em Groups Geom. Dyn.}, 4(2):377--405, 2010.

\bibitem{watatani:kazhdan_implies_fa82}
Yasuo Watatani.
\newblock Property {T} of {K}azhdan implies property {FA} of {S}erre.
\newblock {\em Math. Japon.}, 27(1):97--103, 1982.

\bibitem{willis:scale_groups}
George Willis.
\newblock Scale groups.
\newblock arxiv:2008.05220, 2022.

\bibitem{wilton:group_actions_on_trees}
Henry Wilton.
\newblock Group actions on trees.
\newblock
  https://www.dpmms.cam.ac.uk/$\tilde{\phantom{a}}$hjrw2/Talks/trees.pdf, 2004.

\bibitem{woess:rw}
Wolfgang Woess.
\newblock {\em Random walks on infinite graphs and groups}, volume 138 of {\em
  Cambridge Tracts in Mathematics}.
\newblock Cambridge University Press, Cambridge, 2000.

\end{thebibliography}

\def\cprime{$'$} \def\cydot{\leavevmode\raise.4ex\hbox{.}} \def\cprime{$'$}
  \def\cprime{$'$} \def\cprime{$'$} \def\cprime{$'$} \def\cprime{$'$}
  \def\cprime{$'$} \def\cprime{$'$} \def\cprime{$'$} \def\cprime{$'$}
  \def\cprime{$'$} \def\cprime{$'$} \def\cprime{$'$} \def\cprime{$'$}
  \def\cprime{$'$}

\end{document}